\documentclass[reqno,12pt]{amsart}
\usepackage{geometry}
\usepackage{amsmath,amssymb,amsthm,mathrsfs,color,lineno,paralist,graphicx,float}
\usepackage[colorlinks,
linkcolor=red,
anchorcolor=green,
citecolor=blue, 
]{hyperref}

\usepackage[T1]{fontenc}
\usepackage[utf8]{inputenc}

\usepackage{calc}
\newtheorem{THEalpha}{Theorem}
\newtheorem{PROalpha}[THEalpha]{Proposition} 
\newtheorem{CORalpha}[THEalpha]{Corollary}

\newtheorem{The}{Theorem}[section]
\newtheorem{Cor}{Corollary}[section]
\newtheorem{Lem}{Lemma}[section]
\newtheorem{Pro}{Proposition}[section]
\theoremstyle{definition}
\newtheorem{defn}{Definition}[section]
\theoremstyle{remark}
\newtheorem{Rem}{Remark}[section]
\newtheorem{ex}{Example}[section]

\numberwithin{equation}{section} 

\newcommand{\TT}{\mathbb{T}}
\newcommand{\RR}{\mathbb{R}}
\newcommand{\ZZ}{\mathbb{Z}}
\newcommand{\NN}{\mathbb{N}}

\newcommand{\FF}{\mathbf{F}}
\newcommand{\KK}{\mathbf{K}}
\newcommand{\cF}{\mathcal{F}}
\newcommand{\cK}{\mathcal{K}}
\newcommand{\cR}{\mathcal{L}}
\newcommand{\cG}{\mathcal{G}}
\newcommand{\cE}{\mathcal{E}}
\newcommand{\cU}{\mathcal{U}}
\newcommand{\cD}{\mathcal{I}}

\newcommand{\bff}{\mathbf{f}}
\newcommand{\bfk}{\mathbf{k}}
\newcommand{\bfh}{\mathbf{h}}

\newcommand{\rS}{\mathrm{S}}
\newcommand{\rR}{\mathrm{R}}
\newcommand{\rD}{\mathcal{I}}
\newcommand{\vep}{\varepsilon}

\newcommand{\DelU}{\Delta_{U_0}}
\newcommand{\Diff}{\textup{Diff}_0^\infty}
\def\lrn#1{\left\|#1\right\|}

\def\leq{\leqslant}
\def\geq{\geqslant}
\def\tilde{\widetilde}

\title[]{On simultaneous linearization of certain commuting nearly integrable diffeomorphisms of the cylinder}

\author{Qinbo Chen}
\address{Department of Mathematics, Kungliga Tekniska H\"{o}gskolan, Lindstedtsv\"{a}gen 25, SE-100 44  Stockholm, Sweden}
\email{qinbochen1990@gmail.com}
\author{Danijela Damjanovi\'{c}}
\address{Department of Mathematics, Kungliga Tekniska H\"{o}gskolan, Lindstedtsv\"{a}gen 25, SE-100 44  Stockholm, Sweden}
\email{ddam@kth.se}
\author{Boris Petkovi\'{c}}
\address{Department of Mathematics, Kungliga Tekniska H\"{o}gskolan, Lindstedtsv\"{a}gen 25, SE-100 44  Stockholm, Sweden}
\email{borisp@kth.se}


\begin{document} 
\begin{abstract}
Let  $\mathcal{F}$ and $\mathcal{K}$ be commuting $C^\infty$ diffeomorphisms of the cylinder $\mathbb{T}\times\mathbb{R}$ that are, respectively, close to $\mathcal{F}_0 (x, y)=(x+\omega(y), y)$ and $T_\alpha  (x, y)=(x+\alpha, y)$, where $\omega(y)$ is non-degenerate and $\alpha$ is Diophantine.  
Using  the KAM iterative scheme for the group action we show that $\mathcal{F}$ and $\mathcal{K}$ are simultaneously $C^\infty$-linearizable
 if $\mathcal{F}$ has the intersection property (including the exact symplectic maps) and $\mathcal{K}$ satisfies a  semi-conjugacy condition. 
 We also provide examples showing necessity of these conditions.
 As a consequence, we get local rigidity of certain class of $\mathbb{Z}^2$-actions on the cylinder, generated by commuting twist maps.
\end{abstract}
\maketitle

\section{Introduction}
The goal of this paper is to study the simultaneous linearization problem for some commuting nearly integrable $C^\infty$ diffeomorphisms of the cylinder. The question of linearization has been one of the central themes in dynamical systems. 
We start by considering two types of typical integrable maps on the infinite cylinder $\TT\times\RR$,
whose perturbations will be discussed  below. Here, $\TT=\RR/\ZZ$ denotes the circle. 
Let $\cF_0:\TT\times\RR\to \TT\times\RR$ be a smooth integrable twist map of the form
\[\cF_0 (x, y)=(x+\omega(y), y),\]
where the frequency map $\omega(y)$ is non-degenerate, in the sense that  $\omega(y):\RR\to\RR$ has a smooth inverse map. 
A typical example is $\omega(y)=y$.  
For $\alpha\in\RR$, we denote by $T_\alpha: \TT\times\RR\to \TT\times\RR$   the  linear map as follows
\[T_\alpha  (x, y)=(x+\alpha, y).\]
Clearly,  the phase spaces of $\cF_0$ and $T_\alpha$ are completely foliated by smooth invariant circles,  on which the dynamics are conjugate to the rigid rotations.

We wish to study the  perturbations of $\cF_0$ and the perturbations of $T_\alpha$. They arise naturally in many physical and geometric problems. 
Consider a smooth  diffeomorphism (not necessarily symplectic) $\cF$ which is  a  perturbation of $\cF_0$ and homotopic to the identity. This means there is a perturbation $f=(f_1, f_2)$ with $f_1, f_2\in C^\infty(\TT\times\RR,\RR)$, such that 
\begin{align}\label{form_cF}
	\cF=\cF_0+f &~:~  \TT\times\RR\to \TT\times\RR \nonumber\\
	&\big(x,y\big)
	\longmapsto
\big(x+\omega(y) +f_1(x,y)\quad\textrm{mod}~1, \quad y+f_2(x,y)
	\big).
\end{align} 
In particular, for the case where $\cF$ is exact symplectic, 
 the question of persistence of invariant circles has been much studied.
The celebrated  KAM (Kolmogorov-Arnold-Moser) theorem asserts that  the  Diophantine invariant circles persist under small perturbations. 
Moreover, the question of when there do or do not exist invariant circles has led to deep studies by R\"ussmann, Herman, Mather, \textit{et al.} See \cite{Herman_1986} and the references therein. 

We also consider a perturbation $\cK$ of $T_\alpha$ that is homotopic to the identity. This means there is a perturbation $k=(k_1, k_2)$ with $k_1, k_2\in C^\infty(\TT\times\RR,\RR)$, such that 
\begin{align}\label{form_cK}
	\cK=T_\alpha+k &~:~  \TT\times\RR\to \TT\times\RR \nonumber\\
&\big(x,y\big)
	\longmapsto
\big(x+\alpha+k_1(x,y)\quad \textrm{mod}~1,\quad
		y+k_2(x,y)\big).
\end{align}
There are  many related problems and results under certain assumptions. We briefly review some of them.
Restricted to the bounded annulus $\TT\times[0,1]$,
 an important model is the irrational pseudo-rotation, i.e.,  an orientation and area preserving diffeomorphism of the annulus that  has no periodic points and its rotation number on a boundary circle  is $\alpha$. For any Liouville $\alpha\in\RR$, examples of weak mixing pseudo-rotations  were constructed by the Anosov-Katok method \cite{Anosov_Katok1970,Fay_Sap_2005}.
For Diophantine $\alpha$, it was observed by Herman (see also \cite{Fayad_Krikorian}) that the pseudo-rotation is always smoothly conjugated to $T_\alpha$ in a small neighborhood of the boundary circle. 
We also refer to  course note  \cite{Cro_06} and a recent work
\cite{Av_Fa_Le_Xu_Zh} for more background and overview of the properties of the pseudo-rotations. However, our paper does not focus on the pseudo-rotations. In fact, we do not presuppose  the existence of a $\cK$-invariant circle and the area-preserving condition  for the  map $\cK$. 

In this paper, we are  interested in the \textit{local rigidity} aspect of $\cF_0$ and $T_\alpha$, i.e.,
 the preservation of smooth foliations  under small perturbations. This is essentially a linearization problem.
In general, it is not possible to find a  smooth conjugacy to the linear model for a single element of the action generated by the pair $(\cF, \cK)$.
Indeed, for a single map, 
 it has been known since the work of  Poincar\'e that the smooth foliation structure is in general  destroyed by 
 an arbitrarily small perturbation.  
 Here, we are motivated by an attempt to  investigate the following question:
 
 \noindent\textbf{Question.}\textit{ For the smooth cylinder maps $\cF$ and $\cK$ that are, respectively, close to $\cF_0$ and $T_\alpha$,  assume that $\cF$ and $\cK$ commute (i.e., $\cF\circ\cK=\cK\circ\cF$), 
can  $\cF$ and $\cK$   be  simultaneously $C^\infty$-linearizable ?}

The present paper gives a positive answer in the case where $\alpha$ is Diophantine. 
 
 The linearization problem for commuting diffeomorphisms is 
 related to the rigidity theory of a higher rank $\ZZ^n$-action  where $n\geq 2$ is the number of diffeomorphisms which generate the action. The case of circle maps has been thoroughly stuided. In \cite{Moser90_commuting},
  the problem of linearizing commuting circle diffeomorphisms was raised by Moser  in connection with the holonomy group of certain foliations with codimension $1$. Using a perturbative KAM scheme, he proved that  for commuting $C^\infty$ circle diffeomorphisms $\phi_1,\cdots, \phi_n$, if  the rotation numbers satisfy a simultaneous Diophantine condition and $\phi_1,\cdots, \phi_n$ are  close to the rigid rotations, then they can be simultaneously $C^\infty$-conjugated to the rigid circle rotations.  Later,
the global version of  Moser's result was  proved by  Fayad and Khanin
  \cite{Fayad_Khanin} by using the global theory of Herman \cite{Herman_1979circle} and Yoccoz \cite{Yoccoz_1984}.
  In the higher dimensional case,  the local rigidity  for commuting diffeomorphisms (close to the torus translations) of  $\TT^d$ was obtained in \cite{Rodriguez_2005} for $d=2$ and in  \cite{Damjanovic_Fayad,Wilkinson_Xue, Boris_2021} for  $d\geq 2$,  by assuming an appropriate Diophantine condition on the rotation sets.
  
  Historically, the dynamical motivation  for investigating the rigidity of abelian group actions started with the study of structural stability of hyperbolic diffeomorphisms, see \cite{KS_1997} for a brief introduction. Unlike the rigidity of elliptic group actions which mainly uses  analytic methods, rigidity of the hyperbolic group actions uses more geometric techniques from the hyperbolic theory. For higher rank Anosov actions on compact manifolds,   the rigidity problem has been widely studied (cf. \cite{Hur_1992, Kal_Sa_2006,FKS_2013,RW_2014,DX_2020}, \textit{etc}).
For local rigidity of certain
  higher rank partially hyperbolic abelian actions, see \cite{Damjanovic_Katok10, DK_2011, Damjanovic_Fayad, Vin_Wang_2019} and  the references therein.   A complete local picture for affine actions by higher rank lattices in semisimple Lie groups was obtained in \cite{FM_2009}. For background and  overview of the local rigidity problem for general group actions, we refer to the survey \cite{Fishier_2007Local}.

 The linearization problem in this paper is  inspired by studying a corresponding local rigidity question for a class of parabolic  $\ZZ^2$-action on the cylinder $\TT\times\RR$. 
More precisely,
 consider a $\ZZ^2$-action $(G_1, G_2)$ generated by two  linear twist  maps  $G_1(x,y)= (x+y+\alpha_1, y)$ and $G_2(x,y)=(x+y+\alpha_2,y)$. 
Let $(\tilde G_1, \tilde G_2)$ be a small perturbation of the action $(G_1, G_2)$. Then, analogous to \cite{Moser90_commuting} one may assume  that  the frequency maps satisfy a simultaneous Diophantine condition as follows, 
\[
 \max_{j=1,2}\left|e^{i2\pi m(y+\alpha_j)}-1\right|\geq\frac{\sigma}{|m|^\tau} ,\quad \forall~m\in \ZZ\setminus\{0\},\quad \forall ~ y\in \RR.
 \]
Nevertheless, this also implies that the number $\alpha_2-\alpha_1$ is Diophantine ( by taking $y=-\alpha_1$ in the inequality above). Meanwhile, observe that  the diffeomorphisms $\tilde G_1$ and $\tilde G_2\circ\tilde G_1^{-1}$ are also  generators of the  action $(\tilde G_1, \tilde G_2)$, so the local rigidity problem of $(\tilde G_1, \tilde G_2)$ is equivalent to that of 
 $(\tilde G_1, \tilde G_2\circ\tilde G_1^{-1})$. 
Thus, one finds that $\tilde G_1$ is of the form \eqref{form_cF}, and
  $\tilde G_2\circ\tilde G_1^{-1}$ is a small perturbation of the translation map $T_{\alpha_2-\alpha_1}$, so it is of  the form \eqref{form_cK}. Consequently, it reduces to  the simultaneous linearization problem of commuting 
 diffeomorphisms $\cF$ and $\cK$ given in \eqref{form_cF}--\eqref{form_cK}.

\subsection{Statement of results}
Denote by $\Diff(\TT\times\RR)$ the set of $C^\infty$ diffeomorphisms of the infinite cylinder $\TT\times\RR$ that are homotopic to the identity. The  diffeomorphisms $\cF$ and $\cK$ defined in \eqref{form_cF}--\eqref{form_cK} belong to  $\Diff(\TT\times\RR)$.

A number $\alpha\in\RR$ is said to be \emph{Diophantine} if there exist $\tau>0$ and $\sigma>0$ such that  \begin{equation}\label{def_Dioph}
 \left|e^{i2\pi m\alpha}-1\right|\geq\frac{\sigma}{|m|^\tau }\qquad \forall~m\in \ZZ\setminus\{0\}.
 \end{equation}
In the sequel, we denote by  $\mathrm{DC}(\sigma,\tau)$ the set of all numbers satisfying \eqref{def_Dioph}.

\begin{defn}\cite{Siegel_Moser1971}
	A map $F(x,y)$ of $\TT\times\RR$ is said to satisfy the \emph{intersection property} if each homotopically nontrivial circle $C^1$-close
	to $\{y=const\}$ intersects its image under $F$.
\end{defn}
 
 \begin{Rem}
 	It is known that    
 any exact symplectic map of  $\TT\times\RR$ 
 has the intersection property.   In addition,  an area-preserving map of  $\TT\times \RR$ having at least one homotopically nontrivial invariant circle  also satisfies such a property. Here, we mention that the intersection property was also  used to obtain the KAM-type result (e.g. codimension-one invariant tori) for certain non-symplectic maps of $\TT^d\times\RR$, $d\geq 1$ \cite{Cheng_Sun1989, Xia_1992, Yoccoz_1992}.
 \end{Rem}

We need the notion of semi-conjugacy.  
\begin{defn}
	 For the map $\cK$ defined in \eqref{form_cK}, we say $\cK$ is \emph{Lipschitz semi-conjugate} to the rigid circle rotation $R_\alpha$  $: x\mapsto x+\alpha$ mod $\ZZ$ if there exists a Lipschitz continuous surjective  map $W: \TT\times\RR\to \TT$  such that  $W \circ \cK=R_\alpha \circ W$. 
The Lipschitz semi-conjugacy $W$ can be written as  
 $W(x,y)=x+v(x,y)$ mod $\ZZ$ for some function $v\in \textup{Lip}(\TT\times\RR,\RR)$. 
\end{defn} 

For example,  $K(x,y)=(x+\alpha,y+k_2(x,y))$ is always Lipschitz semi-conjugate to the  rotation $R_\alpha$ via the projection map $\pi_1(x,y)=x$, that is
 $\pi_1 \circ K=R_\alpha \circ \pi_1$.

Throughout this paper,  the frequency map $\omega(y)$ in  $\cF_0(x,y)=(x+\omega(y),y)$
is always assumed to be \textbf{non-degenerate}, in the sense that 
$\omega(y): \RR\to\RR$ is a smooth diffeomorphism of $\RR$. This also implies  that $\cF_0$ is smoothly conjugate to the standard twist map $(x,y)\mapsto (x+y, y)$, see Section \ref{section_renorm}. 
We are now ready to state the main result.
 \begin{THEalpha}\label{MainR1}
 Let  $\cF,  \cK \in \Diff(\TT\times\RR)$ be commuting diffeomorphisms 
 defined as in \eqref{form_cF} and \eqref{form_cK}.
  Suppose that $\cF$ satisfies the intersection property and $\cK$ is Lipschitz semi-conjugate to the rigid rotation  $R_\alpha$ $: x\mapsto x+\alpha$ mod $\ZZ$ with $\alpha\in \mathrm{DC}(\sigma,\tau)$.
 
 Then, there exists $\mu=\mu(\tau)>0$ such that:  
for any $\delta>0$ and  bounded open interval $\rD\subset \RR$ we denote by
 $\mathcal{I}_\delta=\{y\in\RR:~\textup{dist}(y, \mathcal I)<\delta\}$  the $\delta$-neighborhood of $\mathcal I$, if
\[\left\|\cF-\cF_0\right\|_{C^\mu(\TT\times \rD_\delta)}<\vep_0,\qquad\left\|\cK-T_\alpha\right\|_{C^\mu(\TT\times \rD_\delta)}<\vep_0\]
for a sufficiently small  $\vep_0=\vep_0(\tau, \mathcal{I}, \delta)>0$,
then $\cF$ and $\cK$ can be simultaneously $C^\infty$-conjugated to $\cF_0$ and $T_\alpha$ on $\TT\times\rD$, in the sense that there exists a  $C^\infty$ diffeomorphism $H$ from $\TT\times\rD$ onto its image such that 
 \begin{align*}
 	H^{-1}\circ \cF \circ H=\cF_0,\qquad
 	H^{-1}\circ \cK \circ H=T_\alpha.
 \end{align*}
\end{THEalpha}

\begin{Rem}
  It is worth noting that  we do not presuppose the intersection property for $\cK$. We also do not presuppose the existence of any invariant circle for  $\cK$.
\end{Rem}

\begin{Rem}
Even if $\cF$ and $\cK$ are assumed to be both symplectic, the conjugacy $H$ is in general non-symplectic.
	  For example, let  $\cF_0(x,y)=(x+y,y)$ and we consider the perturbations 
	  $\cF(x,y)=(x+y+\vep y, y)$, with $\vep\ll 1$, and $\cK=T_\alpha$. Then, we can define   \[H: (x,y)\mapsto( x, \frac{y}{1+\vep} ).\] 
	  It is easy to check that $H^{-1}\circ\cF\circ H(x,y)=(x+y,y)=\cF_0$ and 
	$H^{-1}\circ\cK\circ H(x,y)=T_\alpha$. Obviously, $H$ is non-symplectic if $\vep\neq 0$.
\end{Rem}

\begin{Rem}  
  As we will see in Section \ref{section_Mproof}, for the value $\mu=\mu(\tau)$
 	 it is enough to take any number greater than or equal to $15(\varrho+1)=15([\tau]+3)$.  
 \end{Rem}

We point out that without  the  semi-conjugacy condition our simultaneous linearization result is not true in general.
The following existence result shows that using only  
   the intersection property of $\cF$ and the commutativity condition can not guarantee  that both maps are  linearizable.

\begin{PROalpha}\label{main_pro_1}
Let $\alpha\in\RR$, and $\mathcal{I}$ be a bounded open interval and $\delta>0$.
For any  $r\in \NN$ and any small $\vep>0$, we can always find  two commuting diffeomorphisms  $\cF,  \cK \in \Diff(\TT\times\RR)$ where  $\cF$
  satisfies the intersection  property and 
 \[\left\|\cF-\cF_0\right\|_{C^r(\TT\times \mathcal{I}_\delta)}<\vep,\qquad\left\|\cK-T_\alpha\right\|_{C^r(\TT\times \mathcal{I}_\delta)}<\vep,\]
 but at least one of the maps
 $\cF$ and $\cK$ is non-integrable in $\TT\times \mathcal{I}$.
\end{PROalpha}

In Proposition \ref{main_pro_1} there is no restriction on the number $\alpha$  (not necessarily irrational or Diophantine).  
 
As a direct corollary of Theorem \ref{MainR1} and Proposition \ref{main_pro_1}, we obtain a result concerning perturbations of abelian actions generated by integrable twist maps.  

We say that two maps $\tilde G_1, \tilde G_2$ of the cylinder are $\alpha$-\emph{compatible} if $\tilde G_2\circ\tilde G_1^{-1}$ is Lipschitz semi-conjugate to the rigid circle rotation $x\mapsto x+\alpha$. 
 
 \begin{CORalpha}\label{coro} 
Consider two  linear twist  maps  $G_1(x,y)= (x+y+\alpha_1, y)$ and $G_2(x,y)=(x+y+\alpha_2,y)$ with $\alpha:= \alpha_2-\alpha_1$ a Diophantine number. Let $\tilde G_1, \tilde G_2$ be commuting  diffeomorphisms of the cylinder which are
  $\alpha$-compatible, and such that some element of the action $(\tilde G_1, \tilde G_2)$ has the intersection property.  If $\tilde G_1, \tilde G_2$ are sufficiently close to $G_1, G_2$, respectively, then $\tilde G_1, \tilde G_2$ are  simultaneously smoothly conjugate to $G_1, G_2$, respectively. Consequently, all the elements of the action $(\tilde G_1, \tilde G_2)$ are integrable. 
  
 Without the $\alpha$-compatibility condition, there exist perturbations  $\tilde G_1, \tilde G_2$  such that the action $(\tilde G_1, \tilde G_2)$ contains a non-integrable element. 
\end{CORalpha}

As a by-product of 
 Theorem \ref{MainR1}, one can study the local perturbation problem of certain commuting non-ergodic maps of $\TT^2$. Recall that an automorphism of $\TT^2$ is determined by a matrix in $\textup{GL}(2,\ZZ)$ with determinant $\pm 1$. Given a matrix  $A=
\left(
\begin{matrix}
	1  & n\\
	0 & 1
\end{matrix}
\right)$ with $n\in\ZZ\setminus\{0\},$
it determines a  toral automorphism which we also denote by $A$. 
 In other words,  
 \[A: \TT^2\to \TT^2,   \qquad (x, y)\mapsto (x+ny~\textup{mod}~ 1, y). \]
We also consider a transltion of $\TT^2$ given by $(x,y)\mapsto (x+\alpha, y)$, and  for simplicity we still denote by the same symbol $T_\alpha$. For such toral maps,
it is easy to see that  
$T_\alpha$ is homotopic to the identity while $A$ is not. Moreover,  the commutation relation $A\circ T_\alpha=T_\alpha\circ A$ holds.  

We are interested in the local perturbation problem, so we consider two   perturbations $F_A\in \textup{Diff}^\infty(\TT^2)$ and $G\in \textup{Diff}^\infty(\TT^2)$ which are homotopic to $A$ and $id$ respectively. 
This means that  $F_A$ and $G$ can be defined by $F_A=A+ f$ and $G=T_\alpha+ g$, and $f(x,y), g(x,y)\in C^\infty(\TT^2,\RR^2)$ are both $\ZZ$-periodic in $x$ and $y$.  
Then, under certain assumptions we have the following local result.

 \begin{CORalpha}\label{Cor_T2map}
For $\alpha\in \textup{DC}(\sigma,\tau)$,
we can find $\mu=\mu(\tau)>0$ and a small number  $\vep_0=\vep_0(\tau)>0$ such that:  for any pair of commuting maps $F_A\in \textup{Diff}^\infty(\TT^2) $ and $G\in \textup{Diff}^\infty(\TT^2)$, if the
$C^\mu$-distance satisfies
\[\textup{dist}_{C^\mu}(F_A, A)<\vep_0,\qquad \textup{dist}_{C^\mu}(G, T_\alpha)<\vep_0, \]
and if  $F_A$ satisfies the intersection property and $G(x,y)$ is Lipschitz semi-conjugate to the rigid circle rotation $x\mapsto x+\alpha$, then
 there exists a near-identity $C^\infty$ diffeomorphism  $H:\TT^2\to\TT^2$  such that  $H^{-1}\circ F_A\circ H=A$ and  $H^{-1}\circ G\circ H=T_\alpha$.
  \end{CORalpha}

  The proof of Corollary \ref{Cor_T2map} is a direct application of Theorem \ref{MainR1}. Note that $F_A$ (resp. $G$) is necessarily homotopic to $A$ (resp. $id$) because the toral diffeomorphism $F_A$ (resp. $G$) is sufficiently close to $A$ (resp. $T_\alpha$). Thus, 
    the diffeomorphism $F_A$ of $\TT^2$ admits a  natural lift to  the infinite cylinder $\TT\times\RR$, denoted by $\tilde{F_A}$, which can be defined  by $\tilde{F_A}(x,y)=(x+ny+ f_1(x,y)~\textup{mod}~ 1 ,y +f_2(x,y) )$ where $f_1, f_2\in C^\infty(\TT\times\RR, \RR)$  are  periodic in both $x$ and $y$ with period $1$. Meanwhile,  
   the diffeomorphism $G$  also admits a natural lift to  $\TT\times\RR$, denoted by $\tilde{G}$, which can be defined  by $\tilde{G}(x,y)=(x+\alpha+ g_1(x,y)~\textup{mod}~ 1 ,y +g_2(x,y) )$, where $g_1(x,y)$, $ g_2(x,y)$  are   periodic in $x$ and $y$ with period $1$. 
  Therefore,  by applying Theorem \ref{MainR1} with $\cF=\tilde{F_A}$, $\cK=\tilde G$ and $\mathcal{I}=[-1, 1]$ and $\delta=\frac{1}{2}$, we are able to obtain Corollary \ref{Cor_T2map}.

  We also remark that for perturbations of affine abelian actions on the torus $\TT^n$, with parabolic linear parts, there is a more general result on classifying  perturbations \cite{DaFaMa2021}.

\subsection{Remarks  on our assumptions and method}
The assumptions in Theorem \ref{MainR1} are essentially needed for the simultaneous linearization result.  Observe that
we have assumed  three assumptions: (1) the commutativity condition; (2) the intersection property of $\cF$; (3) the Lipschitz semi-conjugacy condition for $\cK$.

$\bullet$ The commutativity condition is important for the simultaneous linearization problem.  For example, consider $\cF(x,y)=\cF_0(x,y)=(x+y,y)$ and $\cK(x,y)=(x+\alpha, y+\vep(x,y) )$ with $\vep(x,y)\neq 0$.  Observe that $\cF$ has the intersection property, and  $\cK$ is obviously Lipschitz semi-conjugate to the circle rotation $x\mapsto x+\alpha$ mod $\ZZ$, but $\cF\circ\cK\neq \cK\circ\cF$.  For this model, 
  it is well known  that for a generic  small  perturbation $\vep(x,y)$,  $\cK$  can not be  conjugated to $T_\alpha$.

$\bullet$ The intersection property of $\cF$ is also crucial, otherwise $\cF$ may not be integrable. 
For example, consider two smooth diffeomorphisms  of the cylinder, 
 $\cF(x,y)=(x+y, y+\varepsilon(y))$ with  $\vep(y)>0$ and $\cK=T_\alpha$. Obviously,  $\cF$ commutes with $\cK$.
  But $\cF$ does not satisfy the intersection property. Then, we find that $\cF$ is non-integrable since there are no  invariant circles.

$\bullet$ The semi-conjugacy condition is also needed (see Proposition \ref{main_pro_1}). In fact,  the Lipschitz semi-conjugacy condition of $\cK$ is only used to control the average part of the perturbation  for $\cK$  during the KAM process. Besides, we also point out that  it is possible to replace  this Lipschitz semi-conjugacy condition by a H\"older semi-conjugacy  whose  H\"older exponent is  close to $1$. See Remark \ref{Rem_holdersemi} for an explanation.

Let us compare our conditions with those used in \cite{Fayad_Krikorian}. 
For a single  diffeomorphism   $\cK$ of the form \eqref{form_cK},  if  the following three conditions (see \cite{Fayad_Krikorian}) are satisfied: 
 \vspace{-0.2in}
 \begin{enumerate}[(i)]
 	\item the intersection property;
 	\item it possesses a smooth invariant circle $\Gamma$ with Diophantine rotation number  $\alpha$;
 	\item it has no periodic points.
 \end{enumerate}
  \vspace{-0.2in}
then $\cK$ can be $C^\infty$-conjugated to  $T_\alpha$ in a small neighborhood $U$ of $\Gamma$.  

If this happens,  and if one continues to assume the commutation relation $\cF\circ\cK=\cK\circ\cF$, then the other map $\cF$ would  also be integrable in the  small neighborhood $U$.  
We briefly explain it here. From the preceding analysis, one  obtains
a diffeomorphism $H_1$ from $U$ onto its image such that $H_1^{-1}\circ \cK\circ H_1=T_\alpha$.  Next, we study the conjugated map $\tilde{\cF}:=H_1^{-1}\circ \cF\circ H_1$. It commutes with $T_\alpha$, and if we write $\tilde{\cF}(x,y)=(x+\omega(y)+f_1(x,y), y+f_2(x,y))$, the commutation relation yields 
$f_1(x+\alpha,y)=f_1(x,y)$ and $ f_2(x+\alpha,y)=f_2(x,y).$ As a consequence, 
  $f_1$ and $f_2$ are supposed to be independent of the variable $x$ since $\alpha$ is Diophantine.  
   In other words, they are of the form $f_1(x,y)=f_1(y)$ and $f_2(x,y)=f_2(y)$.
On the other hand,  $\cF$ also satisfies the intersection property, so  $f_2(y)$ must be $0$.  Thus we obtain $\tilde{\cF}=(x+\tilde\omega(y), y)$, where  $\tilde\omega(y)=\omega(y)+f_1(y)$ would be invertible. Finally, using the transformation $H_2(x,y):=(x, \tilde{\omega}^{-1}\circ \omega(y))$ it  is not difficult to check that
$H^{-1}_2\circ \tilde{\cF}\circ H_2(x,y)=(x+\omega(y), y)=\cF_0(x,y)$ and $H_2^{-1}\circ T_\alpha\circ H_2=T_\alpha$. In conclusion, by setting $H=H_1\circ H_2$, we can conjugate $\cF$ and $\cK$ to $\cF_0$ and $T_\alpha$. 

However, the  above analysis can not be applied to our model since we do not assume the intersection property for the map $\cK$ nor the existence of any $\cK$-invariant circle with rotation number $\alpha$. Instead, for our purpose we impose a semi-conjugacy condition on $\cK$.

Now, we outline  the method for proving Theorem \ref{MainR1}. First, as the frequency map $\omega(y)$ is non-degenerate, 
under a suitable coordinate transformation  Theorem \ref{MainR1} can be reduced to Theorem \ref{Thm_simplified} which studies commuting maps
$\FF=U_0+\bff$ and $\KK=T_\alpha+\bfk$ with $U_0(x,y)=(x+y,y)$. Next, the technique used to prove Theorem \ref{Thm_simplified} is based on a KAM iterative scheme for the group action $(\FF, \KK)$. We linearize the nonlinear problem and solve  the corresponding  linearized equation to obtain a better approximation. By iterating this process, the limit of
 successive iterations  produces a solution to the nonlinear problem.
The commutativity is enough to provide 
a common (approximate) solution  to the linearized conjugacy equations of $(\FF,\KK)$. At each iteration step,
in order to show that the new error is smaller than the initial one, in principle
the hard part is the elimination of the average (over $x\in \TT$)  of the perturbations, i.e. $[\bff]=([\bff_1],[\bff_2])$ and $[\bfk]=([\bfk_1],[\bfk_2])$. For this purpose,  the intersection property of $\FF$ enters and causes the term $[\bff_2]$ to be of higher order, and the semi-conjugacy condition of $\KK$ causes the term $[\bfk_1]$ to be of higher order. Besides, using the commutativity condition we can show that $[\bfk_2]$ is quadratic. As for $[\bff_1]$, this term can be, to some extent, eliminated by choosing suitably an approximate solution to the cohomological equation. 
See Section \ref{section_Leq_Comm} and Section \ref{section_Induct_lem} for more discussions.

\subsection{Structure of this paper}
The paper is organized as follows. 
 Section \ref{section_example} is devoted to prove Proposition \ref{main_pro_1}, the construction is based on the generalized standard family. 
 Section \ref{section_prelim} reviews some basic concepts used in this paper.
 In Section \ref{section_renorm}, by using  a suitable coordinate transformation  we show that the simultaneous linearization problem of $\cF=\cF_0+f$ and $\cK=T_\alpha+k$ are  equivalent to that of  $\FF=U_0+\bff$ and $\KK=T_\alpha+\bfk$, where $U_0(x,y)=(x+y,y)$. Theorem \ref{MainR1} thus reduces to  
 Theorem \ref{Thm_simplified}.  In Section \ref{section_Leq_Comm} and Section \ref{section_Induct_lem}, we study the commutativity property, and prove the inductive lemma which is 
the main ingredient of the iteration process. In Section \ref{section_Mproof}, by applying inductively  Proposition \ref{Pro_iterate} we   
  use the KAM scheme to prove Theorem \ref{Thm_simplified}.  
  
\section{An example of non-integrable commuting diffeomorphisms} \label{section_example}

In this section we prove Proposition \ref{main_pro_1}. 
For this purpose, 
we first introduce 
the generalized standard family. It is a generalization of the Chirikov-Taylor standard family, and is one of the most widely studied family of monotone twist maps.
Consider symplectic diffeomorphisms of the cylinder 
 $\TT\times\RR$ which are defined by
\[S_\vep(x,y)=(x+y+\vep V'(x) ,y +\vep V'(x))\]
where $V(x)\in C^\infty(\TT,\RR)$ is $1$-periodic in $x$. 

$S_\vep$ is a small perturbation of the integrable map  $(x,y)\mapsto (x+y,y)$.   
It is an elementary  fact in symplectic geometry that such a map $S_\vep$  can be induced  by a generating function.
More precisely,  it is implicitly defined by the following  generating function
\begin{align*}
	G(x,X)=\frac{1}{2}(X-x)^2+\vep V(x)
\end{align*}
through the equations:
\[y=-\frac{\partial G(x, X)}{\partial x},\qquad Y=\frac{\partial G(x, X)}{\partial X}.\]
Thus $S_\vep$ is exact symplectic, which implies zero flux, that is 
\[\oint_{S_\vep(\gamma)} ydx=\oint_{\gamma} ydx \] 
for every non-contractible loop $\gamma$ on the cylinder.
As a consequence, $S_\vep$ satisfies the intersection property. 

We also point out that if  $V'(x)$ is $\frac{1}{q}$--periodic with $q\in \NN$, then $S_\vep$ commutes with the linear map $T_{p/q}(x,y)=(x+p/q, y)$ for any $p\in\ZZ$. Indeed,
\begin{align}\label{S_commu_T}
\begin{aligned}
S_\vep\circ T_{p/q}(x,y)=&(x+\frac{p}{q}+y+\vep V'(x+\frac{p}{q}), y+\vep V'(x+\frac{p}{q}) )\\
=&(x+\frac{p}{q}+y+\vep V'(x), y+\vep V'(x) )\\
=&T_{p/q}\circ S_\vep(x,y).
\end{aligned}
\end{align}

Now, we turn to prove Proposition \ref{main_pro_1}. The construction will be based on the generalized standard maps described above.

\begin{proof}[Proof of Proposition \ref{main_pro_1}]
Let $\alpha\in \RR$. For any  $\vep>0$ and any  $r\in \NN$, we can choose a rational number $\frac{p}{q}$ such that 
\begin{equation}\label{alpha_pq}
	 0<|\alpha-\frac{p}{q}|<\vep,
\end{equation}
and choose
$V(x)=\frac{-1}{(2\pi q)^{r+1}}\cos 2\pi qx$. Then we define a pair of smooth diffeomorphisms  
$S_\vep$ and $\cK$ by
\begin{equation}\label{S_vepform}
	S_\vep(x,y)=(x+y+\frac{\vep}{(2\pi q)^r}\sin 2\pi q x, ~y+\frac{\vep}{(2\pi q)^r}\sin 2\pi q x)
\end{equation}
and 
\begin{equation}\label{form_exK}
	\cK(x,y)=(x+\frac{p}{q},y).
\end{equation}
Since $S_\vep$ is exact symplectic, $S_\vep$ satisfies the intersection property. By \eqref{S_commu_T} we see that $S_\vep$ commutes with $\cK$. Moreover, due to \eqref{alpha_pq}--\eqref{S_vepform} the perturbations are small in the $C^r$ topology, 
\[\|S_\vep-S_0\|_{C^r}\leq\vep, \quad \|\cK-T_\alpha\|_{C^r}\leq\vep.\]
However, a basic fact is that there  always exists an arbitrarily small $\vep>0$ such that the generalized standard map $S_\vep$ is chaotic and non-integrable (see an illustration in Figure 
\ref{figure_SM}).

To finish our proof, we recall that the frequency map $\omega(y)$ in $\cF_0$ is a 
smooth diffeomorphism, and its inverse map is denoted by $\omega^{-1}(y)$. Then,
under the coordinate transformation  $Q$ which is defined by
 $Q(x,y)=(x, \omega^{-1}(y))$ and $Q^{-1}(x,y)=(x, \omega(y))$, 
 the map $S_\vep$ can be transformed into
 \begin{align*}
 	\cF=Q\circ S_\vep\circ Q^{-1}=\cF_0+f: \TT\times\RR\to  \TT\times\RR
 \end{align*}
Here, $\cF_0(x,y)=(x+\omega(y), y)$ and $f=(f_1, f_2)$ is given by
\[f_1(x,y)= \frac{\vep}{(2\pi q)^r}\sin 2\pi qx, \quad f_2(x,y)=\omega^{-1}\left(\omega(y)+\frac{\vep}{(2\pi q)^r}\sin 2\pi qx\right)-y.\]
Clearly, on the bounded region $\TT\times\mathcal{I}_\delta$,  $f$ can be arbitrarily small in the $C^r$ topology provided that $\vep$ is small enough.
Moreover, by \eqref{form_exK} we have
$\cK=Q\circ \cK\circ Q^{-1}$. 

Therefore, $\cF$ commutes with $\cK$, and $\cF$ also satisfies the intersection property. In view of the non-integrability of $S_\vep$,
 the desired result follows immediately. 	
\end{proof}

\begin{figure}[H]
	\centering
	\includegraphics[scale=0.2]{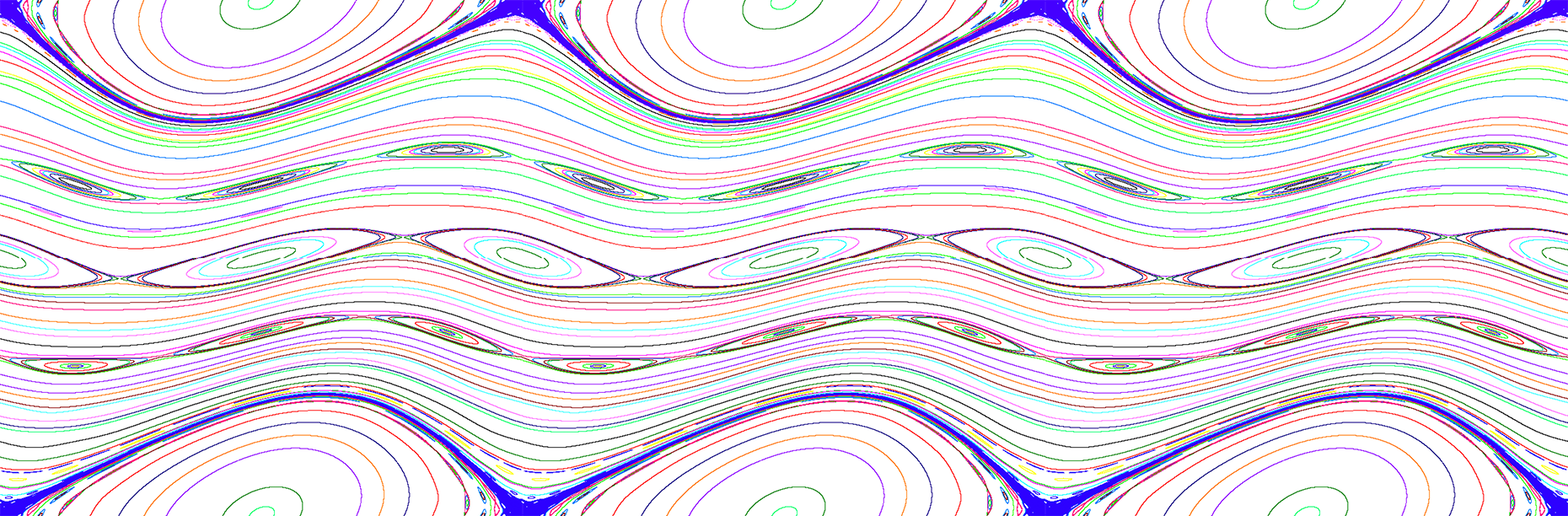}
	\caption{An example for $q=3$.}\label{figure_SM}
\end{figure}

\section{Preliminaries}\label{section_prelim}
	
In this section we review some  basic terminology.

A Fr\'echet space $X$ is defined to be a complete metrizable locally convex topological vector space. Its topology may be induced  by  a family of seminorms $\{\|\cdot\|_r\}_r$. 
A Fr\'echet space $X$ is \emph{graded} if the topology is defined by a family  of semi-norms $\{\|\cdot\|_r\}_r$ satisfying $\|x\|_s\leq \|x\|_t$ for every $x\in X$ and $s\leq t$. For example, the space $C^\infty(\TT^d,\RR)$ with the topology given by the $C^r$ semi-norms $|g|_r=\max_{|j|=r}\sup_{z\in\TT^d}|\partial^j g(z)|$, $r\in \NN$ is a Fr\'echet space. 
By summing up  the first $i$ semi-norms for every $i\in\NN$, it turns  $C^\infty(\TT^d,\RR)$ into a  graded Fr\'echet space.

Our method of this paper shall use some  approximation properties and quantitive estimates, e.g.
the smoothing operators, the interpolation inequalities and the regularity of the composition operator.  
In particular, we need to control the norm of a function in the scale of H\"older spaces. 

Now, let us  turn to define H\"older regularities.
For our purpose, it is sufficient to consider a convex set $U=\TT\times E$ or $\TT$ or $E$, with $E\subset\RR$ an open interval, and then study the H\"older regularities of functions defined on $U$.

For $\lambda\in(0,1)$, we denote by $C^\lambda(U,\RR)$ the space of  bounded $\lambda$-H\"older functions $g:U\to \RR$ with the following norm 
\[ \|g\|_\lambda:=\max\left\{\|g(x)\|_{C^0}, \quad\sup_{0<|z-z'|\leq 1}\frac{|g(z)-g(z')|}{|z-z'|^\lambda}\right\}.\]
For integer $p\in \NN$, $C^p(U,\RR)$ denotes the space of functions with continuous derivatives up to $p$ with the following norm
\[\|g\|_p:=\max_{0\leq t\leq p} \max_{ |j|=t}\,\|\partial^j g\|_{C^0}.\]
For $\ell=p+\lambda$ with $p\in \NN$ and $\lambda\in (0,1)$,  we denote by $C^\ell(U,\RR)$ the space of functions $f:U\to\RR$  with continuous derivatives up to $p$ and H\"older continuous partial  derivatives $\partial^j f$ for all multi-indices $j$ satisfying $|j|=p$. We define its norm by
\[\|g\|_\ell:=\max\left\{\|g\|_p,~\max_{ |j|=p}\|\partial^jg\|_\lambda\right\}.
\]
Here, following  \cite{Sal-Zeh}, we have used  the restriction $0<|z-z'|\leq 1$ for the H\"older part of the norm.  
In this context, an immediate observation is that for any $f\in C^r(U,\RR)$, we have
\[\|f\|_r\geq \|f\|_s,\qquad\text{ for all~} r\geq s\geq 0.\] 
Indeed, this can be readily verified using the mean value theorem, since  the domain $U$ is convex.

In consequence, we find that the space $C^\infty(U,\RR)$ of smooth functions with the family of H\"older norms $\{\|\cdot\|_r\}_{r\geq 0}$ is a graded Fr\'echet space. 

Throughout this paper,
the $C^r$ norm of  a  vector-valued function $G\in C^\infty(U, \RR^l)$ is defined by 
	\[\|G\|_r:=\max_{1\leq i \leq l}\|g_i\|_r,\]
where  $g_i\in  C^\infty(U, \RR)$ is the $i$-th coordinate function of $G=(g_1,\cdots,g_l)$.

\section{Initial reduction}\label{section_renorm}

In this section we will show that the proof of 
Theorem \ref{MainR1} can be reduced to that of Theorem \ref{Thm_simplified}. The basic idea is simple: since  $\omega(y)$ is non-degenerate,
the  map $\cF$ can be transformed into a simplified form which is just a perturbation of the standard integrable map $U_0(x,y)=(x+y,y)$.
	
Recall that the frequency map 
$\omega(y): \RR \longrightarrow \RR$
is a smooth diffeomorphism, with its  inverse denoted by $\omega^{-1}(y)$.
Define  a smooth diffeomorphism $Q$ by
\begin{align*}
	Q: ~\TT\times \RR \longrightarrow \TT\times \RR,\quad
	(x,y) \longmapsto (x, \omega^{-1}(y)),
\end{align*}
and its inverse is
\begin{align*}
Q^{-1}: ~\TT\times \RR \longrightarrow \TT\times \RR,\quad
	(x,y) \longmapsto (x, \omega(y)).
\end{align*}
Under the change of coordinates by $Q$,  the unperturbed map $\cF_0$   can be transformed into
\begin{align*}
	U_0=Q^{-1}\circ\cF_0\circ Q & : ~ \TT\times \RR\longrightarrow \TT\times \RR\\
	 U_0(x,y)&=(x+y,y).
\end{align*} 
Meanwhile, it is easily seen that $T_\alpha$ is invariant under the conjugacy $Q$, that is
\[Q^{-1}\circ T_\alpha\circ Q =T_\alpha.\]

For the maps $\cF$ and $\cK$ considered in Theorem \ref{MainR1}, under the coordinate transformation $Q$ we obtain the corresponding conjugated maps 
\[\FF=Q^{-1}\circ\cF\circ Q :~ \TT\times \RR\longrightarrow \TT\times \RR.\]
	\[ \KK=Q^{-1}\circ\cK\circ Q : ~\TT\times \RR\longrightarrow \TT\times \RR.\]
More precisely,  $\FF=U_0+\bff$ and $\KK=T_\alpha+\bfk$ for some $\bff, \bfk\in C^\infty(\TT\times\RR,\RR^{2})$, and
 \begin{equation}\label{FF_exp}
	\FF(x,y)=(x+y+\bff_1(x,y),~ y+\bff_2(x,y)),
\end{equation}
where
 $\bff_1=f_1\circ Q$ and $\bff_2(x,y)=\omega(\omega^{-1}(y)+f_2\circ Q(x,y))-y$.  
\begin{equation}\label{KK_exp}
	\KK(x,y)=(x+\alpha+\bfk_1(x,y),~ y+\bfk_2(x,y)),
\end{equation}
where $\bfk_1=k_1\circ Q$ and $\bfk_2(x,y)=\omega(\omega^{-1}(y)+k_2\circ Q(x,y))-y$. 

It is easy to verify the following facts.
\begin{Lem}   The commutativity $\FF\circ\KK=\KK\circ\FF$ holds.  
$\FF$ satisfies the intersection property. $\KK$ is Lipschitz semi-conjugate to $R_\alpha$.
\end{Lem}

In fact, the commutativity of $\FF$ and $\KK$ follows directly from that of $\cF$ and $\cK$. The intersection property and the Lipschitz semi-conjugacy property 
are both preserved under coordinate transformations.

 Therefore, by what we have shown above,  Theorem \ref{MainR1} reduces to the following theorem.
 \begin{The}\label{Thm_simplified}
 Let  $\FF, \KK \in \Diff(\TT\times\RR)$ be commuting diffeomorphisms which are induced by $\FF=U_0+\bff$ and $\KK=T_\alpha+\bfk$, where
 $\bff, \bfk\in C^\infty(\TT\times\RR,\RR^{2})$ and $\alpha\in \mathrm{DC}(\sigma,\tau)$. 
 Suppose that  
\begin{itemize}
	\item $\FF$ satisfies the intersection property.
	\item $\KK$ is Lipschitz semi-conjugate to the rigid circle rotation $R_\alpha$.
\end{itemize}
Then, there exists $\mu=\mu(\tau)>0$ such that:  
for any $\delta>0$ and  any  bounded open interval $\cD\subset \RR$,  if the perturbations 
\[\left\|\bff,~\bfk\right\|_{C^\mu(\TT\times\cD_\delta)}<\vep_0\]
for a sufficiently small  $\vep_0=\vep_0(\tau,\mathcal{I},\delta)>0$,
then $\FF$ and $\KK$ can be simultaneously $C^\infty$-conjugated to $U_0$ and $T_\alpha$ on $\TT\times\cD$, in the sense that
there exists a  $C^\infty$ diffeomorphism $H$ from $\TT\times\cD$ onto its image such that 
 \begin{align*}
 	H^{-1}\circ \FF \circ H=U_0,\qquad
 	H^{-1}\circ \KK \circ H=T_\alpha.
 \end{align*}
 \end{The}

We remark that in the above theorem,  $\cD_\delta:=\{y\in\RR,~ \textup{dist}(y, \cD)<\delta\}$. 
For simplicity we have used the notation
\[\|\bff,~\bfk\|_{C^\mu(\TT\times\cD_\delta)}\overset{\textup{def}}=\max\{\|\bff\|_{C^\mu(\TT\times\cD_\delta)}, ~\|\bfk\|_{C^\mu(\TT\times\cD_\delta)}\}.\]
  $C^\infty(\TT\times\RR,\RR^2)$ is the set of  functions $\phi(x,y)\in C^\infty(\RR\times\RR,\RR^2)$ that are  $1$-periodic in $x$.

The following sections will be devoted to prove Theorem \ref{Thm_simplified}.

\section{Linearized conjugacy equations and the commutativity}\label{section_Leq_Comm}

\subsection{Linearized conjugacy equations}
Let us focus on
the commuting diffeomorphisms $\FF=U_0+\bff$ and $\KK=T_\alpha+\bfk$ obtained in \eqref{FF_exp} and \eqref{KK_exp}.
In our setting, the simultaneous $C^\infty$-linearization problem amounts to  
 find a smooth near-identity conjugacy $H=\textup{id}+\bfh$, with  $\bfh(x,y)=(\bfh_1(x,y), \bfh_2(x,y))$ such that
\begin{equation*}
	\KK\circ H=H\circ T_\alpha,\qquad \FF\circ H=H\circ U_0.
\end{equation*}
Since $\KK=T_\alpha+\bfk$, 
the conjugacy equation $\KK\circ H=H\circ T_\alpha$ is reduced to 
\begin{equation}\label{htkH1}
	\bfh\circ T_\alpha-\bfh=\bfk\circ H.
\end{equation}
Simultaneously, as $\FF=U_0+\bff$, the conjugacy equation $\FF\circ H=H\circ U_0$ is reduced to
\begin{align}\label{hfUH2}
\left\{
\begin{array}{lll}
	\bfh_1\circ U_0-\bfh_1-\bfh_2&=\bff_1\circ H\\
	 \bfh_2\circ U_0-\bfh_2&=\bff_2\circ H.
\end{array}
\right.
\end{align}
There is no direct way to solve the nonlinear equations \eqref{htkH1}-\eqref{hfUH2}. 
Instead, we will use a  KAM iterative scheme to solve this nonlinear problem.
In other words, the solution is the limit of successive approximations obtained by approximating the nonlinear problem by its linear part, and solving approximately the corresponding linearized equation.

To simplify the notation, for any convex domain $E\subset \RR$
we define two linear operators on $C^\infty(\TT\times E,\RR^{2})$ as follows: for  $u(x,y)=(u_1(x,y), u_2(x,y))$,
\begin{equation}\label{lin_op_alp}
\begin{aligned}
	\Delta_{\alpha}: ~C^\infty(\TT\times E,\RR^{2})&\longrightarrow C^\infty(\TT\times E,\RR^{2})\\
 u &\longmapsto u\circ T_\alpha-u,
\end{aligned}
\end{equation}
where $T_\alpha(x,y)=(x+\alpha, y)$, and 
\begin{equation}\label{lin_op_U}
\begin{aligned}
	\DelU: ~ C^\infty(\TT\times E,\RR^{2})&\longrightarrow C^\infty(\TT\times E,\RR^{2})\\
	u= \left(\begin{array}{ll}
	u_1\\
		 u_2
	\end{array}
	\right)
	&\longmapsto 
	\left(
	\begin{array}{ll}
	u_1\circ U_0-u_1-u_2\\
	u_2\circ U_0-u_2	
	\end{array}
\right),
\end{aligned}
\end{equation}
where $U_0(x,y)=(x+y,y)$.
It is easily seen that the two linear operators commute, i.e.,
	\[\DelU\Delta_{\alpha}=\Delta_{\alpha}\DelU.\]
Now, the corresponding linearized equations of \eqref{htkH1}--\eqref{hfUH2} can be written as 
\begin{align}
\Delta_\alpha \bfh&=\bfk \label{first_coheq}\\
\DelU  \bfh&=\bff\label{second_coheq}
\end{align} 
where  $\bfh=(\bfh_1, \bfh_2)$, $\bfk=(\bfk_1,\bfk_2)$ and $\bff=(\bff_1,\bff_2)$.

The basic idea of finding a common approximate solution is as follows. 
Thanks to  the Diophantine property of $\alpha$,  one can first  obtain  a  solution $\bfh$ to  equation \eqref{first_coheq}. Then,   by exploiting the commutativity relation  we can show that $\bfh$ also solves equation \eqref{second_coheq} up to a higher order error.  This idea is inspired by Moser's commuting mechanism \cite{Moser90_commuting}.

For our purpose, we first give the following  lemma
 for the linear operator $\Delta_\alpha$. 
  It can be  proved  using Fourier analysis. We repeat the argument here for completeness. We also remark that the norm of the functions are in the scale of H\"older spaces. 
  
\begin{Lem}\label{Lem_estforcoheq}
Let $\alpha\in \mathrm{DC}(\sigma,\tau)$ and $E\subset\RR$ be a convex open set. Given $\varphi(x,y)\in C^\infty(\TT\times E,\RR)$, 
there is a unique solution $u\in C^\infty(\TT\times E,\RR)$ satisfying  $\int_{\TT} u(x,y)\, dx=0$ such that 
\begin{equation}\label{eq_Delalpha}
	\Delta_\alpha u(x,y)=\varphi(x,y)-\int_{\TT} \varphi(x,y)\, dx.
\end{equation}
Moreover,  for all real number $r\in [0,\infty)$ the solution $u$ satisfies
	\begin{align}\label{est_sol_uv}
		\left\|u\right\|_r\leq C\left\|\varphi\right\|_{r+\varrho}, \qquad\varrho=[\tau]+2,
	\end{align}
	where  the constant $C=C(\tau,\sigma)=\frac{1}{\sigma}\sum_{m\neq 0}\frac{1}{|m|^{2+[\tau]-\tau}}$, and  $[\tau]$ is the integer part of $\tau>0$.
	For $r\notin \NN$ we use the H\"older norm (see Section \ref{section_prelim}). 
\end{Lem}
\begin{Rem}
	The constant $C$ can be independent of $\tau$ if one choose $\varrho=\tau+2$ instead of $[\tau]+2$. Sometimes, the linear equation of the form \eqref{eq_Delalpha} is also called a cohomological equation.
\end{Rem}

\begin{proof} 
Using Fourier series  equation \eqref{eq_Delalpha} 
becomes
\begin{align*}
	\sum_{m\in\ZZ\setminus\{0\}}\left(e^{i2\pi m\alpha}-1\right) \widehat{u}_m(y)\, e^{i2\pi mx}=\sum_{m\in\ZZ\setminus\{0\}} \widehat{\varphi}_m(y)\, e^{i2\pi mx}
\end{align*}
where the Fourier coefficients  $\widehat{\varphi}_m(y)=\int_{\TT} \varphi(\theta,y)e^{-i2\pi m\theta}\,d\theta$. 
Then we formally have a solution
\begin{equation*}
	 u(x,y)=\sum_{m\in\ZZ\setminus\{0\}}\frac{\widehat{\varphi}_m(y)}{e^{i2\pi m\alpha}-1}e^{i2\pi mx}.
\end{equation*}
Observe that for each $m$,
\begin{equation}\label{hgdfl}
	 \widehat{\varphi}_m(y)e^{i2\pi mx}=\int_{\TT}\varphi(\theta,y)e^{-i2\pi m(\theta-x)}\,d\theta=\int_{\TT} \varphi(\theta+x,y)e^{-i2\pi m\theta}\,d\theta.
\end{equation}
Using integration by parts, we thus obtain
 \begin{equation}\label{Four_coeffdecay}
  \left\|\widehat{\varphi}_m(y) e^{i2\pi mx}\right\|_p\leq \frac{1}{(2\pi)^q} \frac{\|\varphi\|_{p+q}}{|m|^q}\leq  \frac{\|\varphi\|_{p+q}}{|m|^q},\qquad\text{for all ~}	p, q\in\NN.
 \end{equation}

 Meanwhile,  for each $m$ the following H\"older norm estimate holds
 \begin{equation}\label{Four_coeffdecay_1}
  \left\|\widehat{\varphi}_m(y) e^{i2\pi mx}\right\|_{p+\lambda}\leq  \frac{\|\varphi\|_{p+q+\lambda}}{|m|^q},\qquad\text{for all ~}	p, q\in\NN,\quad \lambda\in(0,1).
 \end{equation}
To verify this estimate, we define $G(x,y)=\widehat{\varphi}_m(y)\,e^{i2\pi mx}$ for simplicity.  Recall that
 \[\left\|G\right\|_{p+\lambda}=\max\left\{\left\|G\right\|_{p},~\max_{|J|=p}\left\|\partial^J G\right\|_{\lambda}\right\}\]
 where the multi-index $J=(J_1, J_2)\in \NN^{2}$ and $\partial^J=\partial ^{J_1}_x\partial^{J_2}_y$. Since  by  \eqref{Four_coeffdecay} 
 \[ \|G\|_p\leq\frac{\|\varphi\|_{p+q}}{|m|^q} \leq  \frac{\|\varphi\|_{p+q+\lambda}}{|m|^q},\] it remains to check the H\"older norm $\left\|\partial^J G\right\|_{\lambda}$ for every multi-index $J$ satisfying $|J|=p$. In fact,
 by \eqref{hgdfl},  for any two points $(x_1,y_1)$ and $(x_2,y_2)$,
\begin{align}
	&\left|\partial^J G(x_1,y_1)-\partial^J G(x_2,y_2)\right|\nonumber\\
	= &\left|\int_{\TT}\left(\partial^J \varphi(\theta+x_1,y_1)-\partial^J \varphi(\theta+x_2,y_2)\right)e^{-i2\pi m\theta}\,d\theta\right|\nonumber\\
	\leq &\frac{1}{|m|^q}\left|\int_{\TT}\left(\partial^{J} \partial_x^q\varphi(\theta+x_1,y_1)-\partial^{J}\partial_x^q \varphi(\theta+x_2,y_2)\right)\,e^{-i2\pi m\theta}\,d\theta\right|\nonumber\\
	\leq &\frac{1}{|m|^q}\,\sup_{\theta}\,\left|\partial^{J}\partial_x^q \varphi(\theta+x_1,y_1)-\partial^{J}\partial_x^q \varphi(\theta+x_2,y_2)\right|
	\label{holder_intbypart}
\end{align}
Here, we have used integration by parts for the third line. As $|J|=p$,
we   infer from  \eqref{holder_intbypart} that
\begin{align*}
	\left\|\partial^J G\right\|_{\lambda}=&\sup\limits_{0<\|(x_1,y_1)-(x_2,y_2)\|\leq 1} \frac{\left|\partial^J G(x_1,y_1)-\partial^J G(x_2,y_2)\right|}{\|(x_1,y_1)-(x_2, y_2)\|^\lambda}\\
	\leq & \frac{1}{|m|^q}\cdot \sup\limits_{0<\|(x_1,y_1)-(x_2,y_2)\|\leq 1} \frac{\|\varphi\|_{p+q+\lambda}\cdot \|(x_1,y_1)-(x_2, y_2)\|^\lambda}{\|(x_1,y_1)-(x_2, y_2)\|^\lambda}\\
	\leq &\frac{1}{|m|^q}\|\varphi\|_{p+q+\lambda}.
\end{align*}
This thus verifies the desired result \eqref{Four_coeffdecay_1}.

Next, we will estimate the $C^r$ norm of the solution $u$ for any $r\geq 0$.
By \eqref{Four_coeffdecay}--\eqref{Four_coeffdecay_1}, for any real $r\in\RR^+$ and $q\in\NN$, 
 \begin{align*}
\left\| u\right\|_r\leq   \sum_{m\in\ZZ\setminus\{0\}}\frac{\left\|\widehat{\varphi}_m(y) e^{i2\pi mx}\right\|_r}{|e^{i2\pi m\alpha}-1|}
\leq  \sum_{m\in\ZZ\setminus\{0\}}\frac{\|\varphi\|_{r+q}}{ |m|^{q}\,|e^{i2\pi m\alpha}-1|}
\leq  \frac{1}{\sigma}\sum_{m\in\ZZ\setminus\{0\}}\frac{\|\varphi\|_{r+q}}{ |m|^{q-\tau}} ,
\end{align*}
where for the last inequality we have used the Diophantine condition  $\alpha\in \mathrm{DC}(\sigma,\tau)$. Note that the series on the right hand side is convergent if and only if  the integer $q$ satisfies $q-\tau>1$. Hence, we can choose $q=[\tau]+2$, then
\begin{align*}
	\left\| u\right\|_r\leq  \frac{1}{\sigma}\sum_{m\in\ZZ\setminus\{0\}}\frac{\|\varphi\|_{r+[\tau]+2}}{ |m|^{2+[\tau]-\tau}} \leq C(\tau,\sigma)\cdot \|\varphi\|_{r+[\tau]+2},
\end{align*}
where the constant $C(\tau,\sigma)=\frac{1}{\sigma}\sum_{m\neq 0}\frac{1}{|m|^{2+[\tau]-\tau}}<\infty$ depends on $\tau$ and $\sigma$.
This therefore proves estimate \eqref{est_sol_uv} for any real $r\geq 0$.
This finishes the proof.
\end{proof}

This lemma tells us that given a differentiable function $\varphi$, the  cohomological equation $\Delta_\alpha u=\varphi-[\varphi]$ has a solution, which  in general is of lower regularity than $\varphi$. However, the loss of regularity  can be controlled by the Diophantine exponent $\tau$. In particular,  the solution $u\in C^\infty$ if $\varphi\in C^\infty$.

\subsection{The commutativity property}

Now we investigate the commutativity assumption.

Suppose that $\FF=U_0+\bff$ commutes with $\KK=T_\alpha+\bfk$  on $\TT\times E$ with $E\subset \RR$ being convex and open. Then
the commutation relation $\FF\circ \KK=\KK\circ\FF$  implies
\begin{equation}\label{come1}
\begin{aligned}
		\bff_1\circ \KK-\bff_1&=\bfk_1\circ\FF-\bfk_1-\bfk_2\\
		\bff_2\circ \KK-\bff_2&=\bfk_2\circ\FF-\bfk_2
\end{aligned}
\end{equation}
on $\TT\times E$.

In view of the linear operators $\Delta_\alpha$ and $\DelU$ defined in \eqref{lin_op_alp}--\eqref{lin_op_U},
we  also introduce  a new linear operator 
\[\cR: C^\infty(\TT\times E,\RR^{2})\times C^\infty(\TT\times E,\RR^{2})
\longrightarrow C^\infty(\TT\times E,\RR^{2})\] 
given by
\begin{equation}\label{crfkform}
\cR(f, g)\overset{\Delta}=\DelU \, g-\Delta_\alpha \, f.
\end{equation}

In what follows, 
for a smooth function $\psi(x,y)$ we use $[\psi](y)$ to denote the average (or mean value) of $\psi$ over $\TT$, that is
 \[[\psi](y)=\int_{\TT} \psi(x,y)\, dx.\]
 In fact,
 this is exactly the $0$-th Fourier coefficient $\widehat\psi_0(y)$ of $\psi(x,y)=\sum\widehat\psi_m(y) e^{i2\pi mx}$.

  For our maps $\FF=U_0+\bff$ and $\KK=T_\alpha+\bfk$, 
 the following result states that  $\cR(\bff, \bfk)$ and the average $[\bfk_2]$  are both of higher order with respect to the size of the perturbations $\bff$ and $\bfk$. This is essentially due to the commutativity property. 
\begin{Lem}\label{Lem_bracket}
If $\FF=U_0+\bff$ commutes with  $\KK=T_\alpha+\bfk$, then the 
 following  estimates hold:
	\begin{align}
	\|\cR(\bff, \bfk)\|_r\leq & C_r \big(\|\bff\|_{r+1}\,\|\bfk\|_r+\|\bfk\|_{r+1}\,\|\bff\|_r\big),\qquad\textup{for any}~r\geq 0\label{cROfk}\\
			\|~[\bfk_2]~\|_0\leq &\|\bff\|_1\,\|\bfk\|_0+\|\bfk\|_1\,\|\bff\|_0\, , \label{k2ave0}	
\end{align}
where $[\bfk_2](y)$ is the average (over $\TT$) of the second component of $\bfk=(\bfk_1,\bfk_2)$.  	
\end{Lem}
\begin{proof}
The commutation relation gives \eqref{come1}, which can be rewritten as
\begin{equation}\label{gfhsb} 
	\DelU \bfk-\Delta_\alpha \bff=\bff\circ\KK-\bff\circ T_\alpha-\bfk\circ\FF+\bfk\circ U_0
\end{equation}
which means that
\begin{align}
	\cR(\bff,\bfk)=&\bff\circ\KK-\bff\circ T_\alpha-\bfk\circ\FF+\bfk\circ U_0
	= \int_0^1 D\bff(T_\alpha+t\bfk)\,\bfk -D\bfk(U_0+t\bff)\,\bff\, dt . \label{dgdfintgral}
\end{align}
Then,
\begin{align*}
	\|	\cR(\bff,\bfk)\|_0 &\leq \|D\bff\|_0\,\|\bfk\|_0+\|D\bfk\|_0\,\|\bff\|_0\leq \|\bff\|_1\,\|\bfk\|_0+\|\bfk\|_1\,\|\bff\|_0.
\end{align*} 
This verifies \eqref{cROfk} for $r=0$. 
Based on \eqref{dgdfintgral}, the $C^r$, $r\geq 1$, norm estimates can be proved similarly, see  for example \cite[Proposition A.2]{Damjanovic_Fayad} or \cite[Appendix II]{Lazutkin_KAM1993}.

Now, it remains to prove inequality \eqref{k2ave0}. Indeed, 
	taking the average over $\TT$ on both sides of  \eqref{gfhsb} we get
	\begin{equation}\label{eq_dntiag}
		[\DelU \bfk]-[\Delta_\alpha \bff]=\int_\TT \bff_1\circ\KK(x,y)-\bff_1\circ T_\alpha(x,y)-\bfk_1\circ\FF(x,y)+\bfk_1\circ U_0(x,y)\, dx.
	\end{equation}	
Here, by  \eqref{lin_op_U} it follows that
\begin{align*}
	[\DelU \bfk]=\left(\begin{array}{ll}
	\int_{\TT}\bfk_1(x+y,y)\, dx-\int_{\TT}\bfk_1(x,y)\, dx-\int_{\TT}\bfk_2(x,y)\, dx\\
	\int_{\TT}\bfk_2(x+y,y)\, dx-\int_{\TT}\bfk_2(x,y)\, dx
\end{array}
\right)=\left(\begin{array}{ll}
	-[\bfk_2]\\
	0
\end{array}
\right).
\end{align*}
Similarly, we can show that $[\Delta_\alpha \bff]=(0,0)$. 
Thus, \eqref{eq_dntiag}  implies that
\[-[\bfk_2](y)= \int_{\TT}\bff_1\circ\KK(x,y)-\bff_1\circ T_\alpha(x,y)-\bfk_1\circ\FF(x,y)+\bfk_1\circ U_0(x,y)\,dx,\]
which yields 
\[\|[\bfk_2]\|_0\leq \|D\bff_1\|_0\|\,\bfk\|_0+\|D\bfk_1\|_0\,\|\bff\|_0\leq\|\bff\|_1\,\|\bfk\|_0+\|\bfk\|_1\,\|\bff\|_0.\]
This finishes the proof.
\end{proof}
 
We end this section by mentioning an interesting result  
\cite{Trujillo_2021} which reveals some connection between the commutativity and the KAM set for the analytic systems. It shows that for two nearly integrable and  exact symplectic $C^\omega$ maps,  if the image of the KAM curves of the two maps intersect on a $C^\infty$-uniqueness set, then the two maps  commute.

\subsection{Smoothing operators}
As we can see from estimate \eqref{est_sol_uv} in Lemma \ref{Lem_estforcoheq}, the $C^r$ norm of the solution $u$ can be estimated by the $C^{r+\varrho}$ norm of $\varphi$, with a fixed loss of regularity $\varrho=[\tau]+2$.
For our KAM iterative scheme in the following sections,  
we shall choose an appropriate smoothing operator to compensate for  this fixed loss of regularity at each iterative step. 
By using interpolation inequalities, one can recover good behavior of some intermediate norms. 
Then the error introduced by this smoothing operator would not destroy the rapid convergence of the iteration (The convergence is not quadratic, but it is still  faster than exponential). This idea comes from the Nash-Moser  technique.

The following  approximation result is well known. We refer to \cite{Moser_rapid66, Zeh_generalized1,Sal-Zeh} for the proof and more details.
\begin{Lem}\label{Lem_trun}
Let $E\subset \RR$ be  open and convex. 
There exists a family of linear smoothing operators $\{\rS_N\}_{N\in\RR^+}$ from $C^\infty(\TT\times E,\RR)$ into itself, such that for every $\psi\in  C^\infty(\TT\times E,\RR)$, one has $\lim_{N\to\infty}\lrn{\psi-\rS_N \psi}_0=0$, and 
\begin{align}\label{trun_est0}
\lrn{\rS_N \psi}_l&\leq C_{s,l} N^{l-s}\lrn{\psi}_s\qquad \text{for~}  l\geq s,
\end{align}
and for the linear operator $\rR_N\overset{\textup{def}}=\textup{id}-\rS_N$, it satisfies 
\begin{align} \label{trun_est1}
		\lrn{ \rR_N  \psi}_s&\leq C_{s,l} \frac{\lrn{\psi}_l}{N^{l-s}} \qquad \text{for~}  l\geq s.
\end{align}
Here,  $C_{s,l}>0$ are  constants depending on $s$ and $l$. 
\end{Lem}

\begin{Rem}
	 In fact,  the smoothing operators $\rS_N$ are constructed by convoluting with appropriate kernels  
  decaying  rather fast at infinity. So, if $\psi$ is periodic in some variables then so are the approximating functions $\rS_N\psi$ in the same variables. Moreover, by the definition of convolution,  it is not difficult to check that
  $[\rS_N\psi](y)=\rS_N[\psi](y)$.
\end{Rem}

However,  the operators $\rS_N$ given in Lemma \ref{Lem_trun} may not preserve the averages, i.e., $[\rS_N\psi](y)\neq [\psi](y)$ and $\rS_N[\psi](y)\neq [\psi](y)$ in general. 

We also note that for the functions defined on $\TT\times E$,   the  Fourier truncation operators $S_N\psi(x,y)$ $=$ $\sum_{|m|\leq N} \widehat{\psi}_m(y) e^{i2\pi mx}$ are not smoothing operators. In fact, for Fourier truncation operators, although inequality \eqref{trun_est1} is still true, 
 inequality \eqref{trun_est0} does not hold for the partial derivatives of $\psi$ with respect to $y$
 (it holds only  for the partial derivatives of $\psi$ with respect to $x$).

As pointed out in \cite{Zeh_generalized1}, one important consequence of the existence of smoothing operators is the  interpolation inequalities (Hadamard convexity inequalities), which will be very useful to us later on.

\begin{Lem}  \cite{Zeh_generalized1}
\label{cor_intpest}
	Let $g\in C^\infty(\TT\times E,\RR)$ with $E\subset \RR$ convex and open. Then, for all  $s\leq m\leq l$, $m=(1-\lambda) s+\lambda l$ with $\lambda\in[0,1]$,
	\[\lrn{g}_m\leq C_{\lambda,s,l}\lrn{g}_s^{1-\lambda}	\lrn{g}_l^{\lambda},\] 
 where the constants $C_{\lambda,l,s}>0$ depend only on $l, s$ and $\lambda$.
\end{Lem}
In fact, as $s\leq m\leq l$, we choose $N\in\RR^+$ satisfying $N^{l-s}=\frac{ \|g\|_l }{ \|g\|_s }$, and then invoke Lemma \ref{Lem_trun} to obtain that
\begin{align*}
	\lrn{g}_m\leq \|\rS_N g\|_m+\|\rR_Ng\|_m\leq & C_{s,m}  N^{m-s}\|g\|_s +C_{m,l} N^{m-l}\|g\|_l\\
	= & (C_{s,m}+C_{m,l}) \|g\|^{\frac{l-m}{l-s}}_s \|g\|_l^\frac{m-s}{l-s}.
	\end{align*}

We also refer to \cite{dlL_Obaya_1999} for a proof done by an elementary method, and extend even to H\"older spaces of functions defined in a Banach space \cite{dlL_tutorial_1999}.

\section{Inductive lemma and the error estimates}\label{section_Induct_lem}

The goal of this section is to prove Proposition \ref{Pro_iterate},
which will be the main ingredient in the proof of Theorem \ref{Thm_simplified}.
 It allows us to  obtain smaller errors  after each iteration, which thus ensures the convergence of our KAM iteration scheme, see Section \ref{section_Mproof}.

Let  $\alpha\in \mathrm{DC}(\sigma,\tau)$,  we  recall the constant 
\[\varrho=[\tau]+2\]
obtained in Lemma \ref{Lem_estforcoheq}.
Then the following result holds.
\begin{Pro}\label{Pro_iterate}
Let
 $\FF=U_0+\bff$ and $\KK=T_\alpha+\bfk$ be commuting $C^\infty$ diffeomorphisms, where  $\FF$ has the intersection property.  
Let $\delta\in(0, \frac{1}{2}]$ and $\cD\subset\RR$ be a bounded open interval, we write $\|\bff,~\bfk\|_r=\|\bff,~\bfk\|_{C^r(\TT\times\cD_\delta)}$. Suppose that $\KK$ is semi-conjugate to $R_\alpha$ via a Lipschitz semi-conjugacy of the form $W(x,y)=x+v(x,y)$, where  $v\in \textup{Lip}(\TT\times\cD_\delta,\RR)$ satisfies  
 $|v(z)-v(z')|\leq \mathfrak L\cdot \textup{dist}(z, z')$ for 
some $\mathfrak L>1$.

Then, for $N>1$, there exists  $\bfh \in C^\infty(\TT\times\cD_\delta,\RR^{2})$, see formula \eqref{def_solh}, satisfying
\begin{equation}\label{hhhhnorm}
	\|\bfh\|_r\leq  C_{r',r,\varrho}\, N^{r-r'+\varrho}\| \bff,~\bfk\|_{r'},\qquad \text{for~}r\geq r'\geq 0.
\end{equation}
Denote $\theta=\|\bfh\|_1$, $\theta'=\|\bff,~\bfk\|_0$ and assume that
\begin{equation}\label{newdelta}
	\widetilde{\delta}:=\delta-2\theta-\theta'>0,
\end{equation}
then the map $H=\textup{id}+\bfh$ has a smooth inverse
$H^{-1}$ defined on $\TT\times\cD_{\delta-\theta}$, and the conjugated maps
\begin{equation*}
	\widetilde\FF=H^{-1} \circ\FF \circ H,  \qquad 	\widetilde\KK=H^{-1}\circ\KK \circ H,
\end{equation*}
are smooth diffeomorphisms from $\TT\times\cD_{\widetilde\delta}$ onto their images.

Writing $ \widetilde\FF=U_0+\widetilde \bff$ and $\widetilde\KK=T_\alpha+\widetilde\bfk$, where $\widetilde \bff, \widetilde \bfk\in C^\infty(\TT\times\cD_{\widetilde\delta},\RR^{2})$, we have:
\begin{align}\label{wtf0norm}
	\lrn{\widetilde\bff, ~\widetilde\bfk }_0  \leq & C_{r,\varrho}  \cdot  \mathfrak L \cdot\left(N^{2\varrho}\|\bff, ~\bfk\|_1\,\|\bff,~ \bfk\|_0+\frac{\|\bff,\bfk\|^2_{\varrho+r+1}}{N^r}+\frac{\|\bff, \bfk\|_{\varrho+r}}{N^r}\right), \quad \text{for~}   r\geq 0,
\end{align}
\begin{align}\label{wtf_rnorm}
	\lrn{\widetilde \bff,~\widetilde \bfk}_r\leq & C_{r,\varrho}\,\Big(1+N^{\varrho}\| \bff,~\bfk\|_{r}\Big),\qquad \text{for~}r> 0.
\end{align}
Moreover, $\widetilde \KK$ is semi-conjugate to $R_\alpha$ via a Lipschitz semi-conjugacy $\widetilde W(x,y)=x+\widetilde v(x,y)$, where 
$\widetilde v\in \textup{Lip}(\TT\times\cD_{\widetilde\delta},\RR)$ has a Lipschitz bound $\mathfrak{\widetilde L}>1$ satisfying  
\begin{equation}\label{newbound_L}
	\mathfrak{\widetilde L}\leq \mathfrak L\,(1+2\|\bfh\|_1).
\end{equation}
\end{Pro}
\begin{Rem}
In fact, to simplify the notation we	 have used
 \[\|\bfh\|_r=\|\bfh\|_{C^r(\TT\times\cD_\delta)}, \qquad
\lrn{\widetilde\bff, ~\widetilde\bfk }_r=\lrn{\widetilde\bff, ~\widetilde\bfk }_{C^r(\TT\times\cD_{\widetilde\delta})}.\]
Condition \eqref{newdelta} implies that $\|\bff,~\bfk\|_1$ shall be suitably small.
\end{Rem}

The proof of Proposition \ref{Pro_iterate} will be divided into several lemmas.

\subsection{Construction of $\bfh$}
The following lemma shows that the solution of  the linearized equation $\Delta_\alpha u=\rS_N\bfk-[\rS_N\bfk]$  is, to some extent,  an approximate solution of the linearized equation  $\DelU u=\rS_N\bff-[\rS_N\bff]$. It is essentially due to the commutativity condition (see Lemma  \ref{Lem_bracket}).

For simplicity we introduce the set
\[C_0^\infty(\TT\times \cD_\delta,\RR^{2})=\left\{ \phi(x,y)\in C^\infty(\TT\times \cD_\delta,\RR^{2})~:~ [\phi](y)=\int_{\TT}\phi(x,y)\,dx=0\right\}.\]
\begin{Lem}\label{Lem_xiN}
	Given $N>1$,  there is a unique solution $\xi_N(x,y)\in C_0^\infty(\TT\times \cD_\delta,\RR^{2})$ to the following equation of $u$ 
\begin{equation}\label{Linalpha_eq_N}
	\Delta_\alpha u=\rS_N\bfk-[\rS_N\bfk].
\end{equation}
It satisfies 
\begin{equation}\label{xi_rnorm}
\|\xi_N\|_{r}\leq  C_{r',r+\varrho} N^{r-r'+\varrho}\| \bfk\|_{r'}\, ,
\end{equation}
for any $r\geq r'\geq 0$. Moreover, if we define $\mathcal N$ by
\begin{equation}\label{diffcalN}
	\mathcal{N}(x,y)\overset{\textup{def}}= \DelU \xi_N(x,y)-\big(\rS_N \bff(x,y)-[\rS_N \bff](y)\big).
\end{equation}
Then, 
\begin{equation}\label{est_mathcalN}
	\|\mathcal N\|_0\leq  C_{\varrho,r} \left(N^\varrho\|\bff,~\bfk\|_1\|\bff,~\bfk\|_0+\frac{\|\bff,\bfk\|^2_{\varrho+r+1}}{N^r}+\frac{\|\bff, \bfk\|_{\varrho+r}}{N^r}\right),\quad \textup{for}~ r\geq 0.
\end{equation}
\end{Lem}
\begin{proof}
By Lemma \ref{Lem_estforcoheq}, there is a unique solution denoted by $\xi_N(x,y)\in C_0^\infty(\TT\times \cD_\delta,\RR^{2})$ to the linear equation \eqref{Linalpha_eq_N}, and  
 by estimate \eqref{est_sol_uv}, it follows that $\|\xi_N\|_{r}\leq  C\|\rS_N \bfk\|_{r+\varrho}$.
 Then, due to Lemma \ref{Lem_trun} we have
\begin{equation*}
\|\xi_N\|_{r}\leq   C_{r',r+\varrho} N^{r-r'+\varrho}\| \bfk\|_{r'}\,,
\end{equation*}
for any $r\geq r'\geq 0$.  Next, we consider the function $\mathcal N$. Recall that the smoothing operators $\rS_N$ are constructed by the convolution, 
it is easy to find that every $\rS_N$ commutes with the operator $\Delta_\alpha$, and $\Delta_\alpha$ also commutes with $\DelU$, namely 
\[\Delta_\alpha \DelU=\DelU\Delta_\alpha,\quad \Delta_\alpha \,\rS_N=\rS_N\,\Delta_\alpha.\]
Then $\mathcal N$ satisfies the following equation
\begin{equation}
\begin{aligned}
	\Delta_\alpha \mathcal N=\Delta_\alpha \DelU \xi_N-\Delta_\alpha\rS_N \bff+\Delta_\alpha [\rS_N \bff]
	=& \DelU \Delta_\alpha \xi_N-\Delta_\alpha\rS_N \bff\\
	=& \DelU (\rS_N\bfk-[\rS_N\bfk])-\Delta_\alpha\rS_N \bff\\
	=& \DelU \rS_N\bfk- \DelU [\rS_N\bfk]-\Delta_\alpha\rS_N \bff\\
	=&\cR(\rS_N\bff,\rS_N\bfk)- \DelU [\rS_N\bfk].\label{idDeltacN}\\
	\end{aligned}
\end{equation}
See also   \eqref{crfkform} for the definition of $\cR$. 
Note that  
 the average   
\begin{align*}
	 [\mathcal{N}]=[\DelU \xi_N]=&\left(\begin{array}{ll}
	\int_{\TT}\xi_{N,1}(x+y,y)\, dx-\int_{\TT}\xi_{N,1}(x,y)\, dx-\int_{\TT}\xi_{N,2}(x,y)\, dx\\
	\int_{\TT}\xi_{N,2}(x+y,y)\, dx-\int_{\TT}\xi_{N,2}(x,y)\, dx
\end{array}
\right)\\
=&\left(\begin{array}{ll}
	-[\xi_{N,2}]\\
	0
\end{array}
\right)=\mathbf{0}
\end{align*}
as a result of $\xi_N\in C_0^\infty(\TT\times \cD_\delta,\RR^{2})$.
Thus,  applying  Lemma \ref{Lem_estforcoheq} to \eqref{idDeltacN} we deduce that
\begin{align*}
	\lrn{\mathcal{N}}_0\leq & C\lrn{~\cR(\rS_N\bff,\rS_N\bfk)- \DelU [\rS_N\bfk]~}_\varrho\,.
\end{align*}
Since $\rS_N=id-\rR_N$, we have
\begin{align*}
	\cR(\rS_N\bff,\rS_N\bfk)=&\cR(\bff,\bfk)-\Delta_{U_0}\rR_N\bfk+\Delta_{\alpha}\rR_N\bff\\
	=& \rS_N\cR(\bff,\bfk)+ \rR_N\cR(\bff,\bfk)-\Delta_{U_0}\rR_N\bfk+\Delta_{\alpha}\rR_N\bff.
\end{align*}
Thus,  by Lemma \ref{Lem_trun} and inequality \eqref{cROfk} of Lemma  \ref{Lem_bracket} we deduce that: for any $r\geq 0$,
\begin{align}
	\|\cR(\rS_N\bff,\rS_N\bfk)\|_\varrho	 
	&\leq C_{\varrho,r} \left( N^\varrho\|\cR(\bff,\bfk)\|_0+ \frac{\|\cR(\bff,\bfk)\|_{\varrho+r}}{N^r}+ \|\rR_N\bfk\|_\varrho+\|\rR_N\bff\|_\varrho\right)
	\nonumber\\
	&\leq C'_{\varrho,r} \left( N^\varrho\|\bff,\bfk\|_1\|\bff,\bfk\|_0+\frac{\|\bff,\bfk\|_{\varrho+r+1}\|\bff,\bfk\|_{\varrho+r}}{N^r}+\frac{\|\bff, \bfk\|_{\varrho+r}}{N^r}
 \right) \nonumber\\
 & \leq C'_{\varrho,r}\left( N^\varrho\|\bff,\bfk\|_1\|\bff,\bfk\|_0+\frac{\|\bff,\bfk\|^2_{\varrho+r+1}}{N^r}+\frac{\|\bff, \bfk\|_{\varrho+r}}{N^r}
 \right).
 \label{glhle1}
	\end{align}
Meanwhile, it is easy to check that
   $\DelU [\rS_N\bfk]=(-[\rS_N\bfk_2],0)$, then using Lemma \ref{Lem_trun} and the inequality \eqref{k2ave0} of Lemma  \ref{Lem_bracket},
\begin{align}
	\|\DelU [\rS_N\bfk]\|_\varrho= \|~[\rS_N\bfk_2]~\|_\varrho= \|~\rS_N[\bfk_2]~\|_\varrho 
\leq & C_{\varrho} \,N^{\varrho} \|~[\bfk_2]~\|_{0}\nonumber\\
\leq &	C_\varrho N^{\varrho}\left(\|\bff\|_1\|\bfk\|_0+\|\bfk\|_1\|\bff\|_0\right).\label{qewr2}
\end{align}

Therefore,   \eqref{glhle1} together with \eqref{qewr2} implies that 
\begin{align*}
	\lrn{\mathcal{N}}_0
	\leq & C_{\varrho,r} \left(N^\varrho\|\bff,~\bfk\|_1\|\bff,~\bfk\|_0+\frac{\|\bff,\bfk\|^2_{\varrho+r+1}}{N^r}+\frac{\|\bff, \bfk\|_{\varrho+r}}{N^r}\right),\qquad\textup{for any}~r\geq 0. 
\end{align*}
\end{proof}

Based on the solution $\xi_N$ obtained in Lemma \ref{Lem_xiN}, we  construct the near-identity conjugacy $H=\textup{id}+\bfh$ as follows: 
\begin{equation}\label{def_solh}
	\bfh\overset{\textup{def}}=\left(\begin{array}{r}
	0\\ -{[}\rS_N \bff_1{]}
\end{array}
\right)
+\xi_N=\left(\begin{array}{r}
	\xi_{N,1}\\ -{[}\rS_N \bff_1{]}+\xi_{N,2} 
\end{array}
\right),
\end{equation}
where we write $\xi_N=(\xi_{N,1}, \xi_{N,2})$.
Note that  $\Delta_\alpha [\rS_N \bff_1]=0$,  so $\bfh\in C^\infty(\TT\times \cD_\delta,\RR^{2})$ is still a solution of  \eqref{Linalpha_eq_N}, that is
\[\Delta_\alpha \bfh=\rS_N\bfk-[\rS_N\bfk].\]
 However, the average  $[\bfh]\neq 0$ in general.
 
\begin{Lem}\label{Lem_estim_h}
$\bfh$ satisfies the following estimates. 
\begin{align}
	\|\bfh\|_r
	\leq & C_{r,\varrho} N^{\varrho}\| \bff,~\bfk\|_r\, ,\qquad \text{for every~ } r\geq 0.\label{h_rnorm}\\
	\|\bfh\|_r\leq & C_{r',r,\varrho} N^{r-r'+\varrho}\| \bff, ~\bfk\|_{r'}\,. 
	\qquad \text{for every~ } r\geq r'\geq 0.\label{h_rr_2norm}
\end{align}
Moreover, 
under assumption \eqref{newdelta},
the map $H=\textup{id}+\bfh$  has a smooth inverse 
 \[H^{-1}: \TT\times \cD_{\delta-\theta}\longrightarrow   \TT\times \RR\]
 which is a smooth diffeomorphism from $\TT\times \cD_{\delta-\theta}$ onto its image, and $H^{-1}(\TT\times \cD_{\delta-\theta})$ $\subset$  $\TT\times \cD_{\delta}$.
\end{Lem}
\begin{proof}
Applying Lemma \ref{Lem_trun} and inequality \eqref{xi_rnorm} to the formula \eqref{def_solh},
\begin{equation*}
\begin{aligned}
	\|\bfh\|_r\leq \|\rS_N \bff\|_{r}+\|\xi_N\|_{r}\leq &C_{r',r}N^{r-r'}\|\bff\|_{r'}+C_{r',r+\varrho}N^{r-r'+\varrho}\|  \bfk\|_{r'}\\
	\leq & C_{r',r,\varrho} N^{r-r'+\varrho}\| \bff, ~\bfk\|_{r'}
\end{aligned}
\end{equation*}
for any 
$r\geq r'\geq 0$,
where the constant $C_{r', r,\varrho}>0$ depends on $r', r$ and $\varrho$. This proves the desired estimate \eqref{h_rr_2norm}. 
In particular, \eqref{h_rnorm} follows immediately by taking $r=r'$.

By assumption \eqref{newdelta}, we infer that $\theta=\|\bfh\|_1$ satisfies
\[\theta<\delta/2\leq \frac{1}{4}.\]
Then,   Proposition \ref{Apdix_pro1} implies that the map $H=\textup{id}+\bfh$  has a smooth inverse 
 $H^{-1}$,
 which is a smooth diffeomorphism from $\TT\times \cD_{\delta-\theta}$ onto its image, and $H^{-1}(\TT\times \cD_{\delta-\theta})$ $\subset$  $\TT\times \cD_{\delta}$.
\end{proof}

\subsection{$C^0$-norm estimates of the new errors}
By assumption \eqref{newdelta}, 
 $\theta=\|\bfh\|_1$ and $\theta'=\|\bff, \bfk\|_0$ satisfies
\begin{equation}\label{asp_deltheta}
	\widetilde\delta:=\delta-2\theta-\theta'>0.
\end{equation}
Then, it is easy to find that
 $\FF\circ H (\TT\times \cD_{\widetilde\delta})\subset \TT\times \cD_{\delta-\theta}$. According to Lemma \ref{Lem_estim_h},
 $H^{-1}$ is well defined  on $\TT\times \cD_{\delta-\theta}$, 
 we thus have  the  following conjugated  map 
 \begin{equation*}
 	\widetilde{\FF}=H^{-1}\circ \FF \circ H: ~ \TT\times \cD_{\widetilde\delta}\longrightarrow\TT\times\RR
 \end{equation*}
 which is a smooth diffeomorphism from $\TT\times \cD_{\widetilde\delta}$  onto its image. Similarly,  
 \begin{equation*}
 	\widetilde{\KK}=H^{-1}\circ \KK \circ H: ~ \TT\times \cD_{\widetilde\delta}\longrightarrow\TT\times\RR
 \end{equation*}
 is also a smooth diffeomorphism  from $\TT\times \cD_{\widetilde\delta}
$ onto its image.

We write  $\widetilde{\FF}=U_0+\widetilde{\bff}$ and $\widetilde{\KK}=T_\alpha+\widetilde{\bfk}$, where 
  $\widetilde\bff, \widetilde\bfk \in C^\infty(\TT\times \cD_{\widetilde\delta},\RR^{2})$.
We will show that the new errors $\|\widetilde \bff\|_0$  and $\|\widetilde \bfk\|_0$ are of higher order. As we will see below, the hard part is the average terms.
This is the only place where we need the intersection property and the Lipschitz semi-conjugacy condition. 
  
\begin{Lem}\label{Lem_C0newerror}
For every $r\geq 0$,  
\begin{align}\label{tilbffnom0}
	\lrn{\widetilde\bff}_0	
	\leq C_{r,\varrho}\left(N^\varrho\|\bff,~\bfk\|_1\|\bff,~\bfk\|_0+\frac{\|\bff,\bfk\|^2_{\varrho+r+1}}{N^r}+\frac{\|\bff, \bfk\|_{\varrho+r}}{N^r}\right).
\end{align} 
For $\widetilde \bfk=(\widetilde\bfk_1,\widetilde\bfk_2)$,	it satisfies 
\begin{align}
	 \lrn{\widetilde \bfk_1}_0\leq    C_{r, \varrho}\cdot  \mathfrak L \cdot\left(
N^{2\varrho}\|\bff, ~\bfk\|_1\|\bff, ~\bfk\|_0+\frac{\|\bfk\|_{r}}{N^r}\right),\label{tildeknorm0}\\
	\lrn{\widetilde \bfk_2}_0\leq C_{r, \varrho}\left(
N^{2\varrho}\|\bff, ~\bfk\|_1\|\bff, ~\bfk\|_0+\frac{\|\bfk\|_{r}}{N^r}\right).\label{yokykyp}
\end{align}
Here,  $\mathfrak L>1$ is a Lipschitz bound of  $v(x,y)$ for the semi-conjugacy $W(x,y)=x+v(x,y)$.
\end{Lem}
\begin{proof} 
We first consider $\widetilde \bff$. Note that
 the identity   $
H\circ \widetilde{\FF}=\FF\circ H
$
implies 
\[
\widetilde{\bff}=U_0\circ H-U_0+\bff\circ H-\bfh\circ\widetilde{\FF}.\]
In light of $U_0(x,y)=(x+y, y)$ and $\bfh$ given in \eqref{def_solh}, we deduce that
\begin{align*}
	\widetilde{\bff}=
	&\left(\begin{array}{l}
		\bfh_1+\bfh_2\\
		\bfh_2
	\end{array}\right)
	+\bff\circ H-\bfh\circ\widetilde{\FF}\\
	=& -\DelU \bfh  +\bfh\circ U_0+\bff\circ H-\bfh\circ\widetilde{\FF}\\
	=&-\DelU \bfh +\bff+ (\bff\circ H-\bff+\bfh\circ U_0-\bfh\circ\widetilde{\FF})\\
	=& \left(\begin{array}{l}
		-[\rS_N\bff_1]\\
		0
	\end{array}\right)-\DelU \xi_N+ \bff+(\bff\circ H-\bff+\bfh\circ U_0-\bfh\circ\widetilde{\FF})\\
=& \left(\begin{array}{l}
		0\\
		{[}\rS_N\bff_2{]}
	\end{array}\right)-\left([\rS_N\bff]+\DelU\xi_N-\rS_N\bff\right)+ \rR_N\bff+(\bff\circ H-\bff+\bfh\circ U_0-\bfh\circ\widetilde{\FF})\\
	=& \left(\begin{array}{l}
		0\\
		{[}\rS_N\bff_2{]}
	\end{array}\right)-\mathcal{N}+(\rR_N\bff+\bff\circ H-\bff+\bfh\circ U_0-\bfh\circ\widetilde{\FF}),
\end{align*}
where  $\mathcal N$  is given in \eqref{diffcalN}. Writing 
 $\widetilde \bff=(\widetilde{\bff}_1,\widetilde{\bff}_2 )$ and  $\mathcal{N}=(\mathcal{N}_1, \mathcal{N}_2)$, we get 
\begin{align*}
\begin{array}{lll}
	\widetilde{\bff}_1=& &-\mathcal{N}_1+ \rR_N\bff_1+(\bff_1\circ H-\bff_1)-(\bfh_1\circ \widetilde\FF-\bfh_1\circ U_0),\\
	\widetilde{\bff}_2= &[\rS_N \bff_2]&-\mathcal{N}_2+ \rR_N\bff_2+(\bff_2\circ H-\bff_2)-(\bfh_2\circ \widetilde\FF-\bfh_2\circ U_0).
\end{array}
\end{align*}

Basically, $\tilde\bff_1$ is of higher order. In fact,
we get the following preliminary estimate 
 \begin{align}\label{tildef1norm0}
\begin{aligned}
	\lrn{\widetilde \bff_1}_0\leq & \lrn{\mathcal{N}}_0+ \lrn{\rR_N\bff}_0+\lrn{\bff}_1\lrn{\bfh}_0+\lrn{\bfh}_1\lrn{\widetilde\bff}_0\,.
\end{aligned}
\end{align}
As for $\tilde\bff_2$,  the hard part is 
the average term $[\rS_N\bff_2]$ which is only of  order one  without further information. It is here that the intersection property of $\widetilde \FF$ comes into play,  causing this term to be of higher order. More precisely,
as $\widetilde{\FF}(x,y)=(x+y+\widetilde\bff_1, y+\widetilde\bff_2)$ satisfies the intersection property, we have that for each point $y\in \cD_{\widetilde\delta}$,
\[\big(\TT\times\{y\}\big)  ~\bigcap ~\widetilde\FF\big(\TT\times\{y\}\big) \neq \emptyset, \]
which implies that for every $y$, the map $x\longmapsto \widetilde\bff_2(x,y)$ has  zeros. 
  Hence, it follows that
 \begin{equation}\label{tildebff_2}
 \begin{split}
 	 	\lrn{\widetilde\bff_2}_0\leq & 2\lrn{\widetilde\bff_2-[\rS_N \bff_2]}_0\\
	\leq & 2\Big( \lrn{\mathcal{N}_2}_0+\lrn{ \rR_N \bff_2}_0+\lrn{\bff_2\circ H-\bff_2}_0+\lrn{\bfh_2\circ \widetilde\FF-\bfh_2\circ U_0}_0\Big)\\
	\leq & 2\Big( \lrn{\mathcal{N}_2}_0+\lrn{ \rR_N \bff_2}_0+\lrn{D\bff_2}_0\lrn{\bfh}_0+\lrn{D\bfh_2}_0\lrn{\widetilde \bff}_0\Big)\\
	\leq & 2\Big( \lrn{\mathcal{N}}_0+\lrn{ \rR_N \bff}_0+\lrn{\bff}_1\lrn{\bfh}_0+\lrn{\bfh}_1\lrn{\widetilde \bff}_0\Big).
 \end{split}
 \end{equation}

Since $\lrn{\widetilde\bff}_0=\max\left\{\lrn{\widetilde\bff_1}_0,~\lrn{\widetilde\bff_2}_0\right\}$, we combine \eqref{tildebff_2} with \eqref{tildef1norm0} to obtain
\begin{equation*}
		 	\lrn{\widetilde\bff}_0
	\leq  2\Big( \lrn{\mathcal{N}}_0+\lrn{ \rR_N \bff}_0+\lrn{\bff}_1\lrn{\bfh}_0+\lrn{\bfh}_1\lrn{\widetilde \bff}_0\Big).
\end{equation*}
which yields
\begin{equation*}
		 	(1-2\lrn{\bfh}_1)\cdot\lrn{\widetilde\bff}_0
	\leq  2\Big( \lrn{\mathcal{N}}_0+\lrn{ \rR_N \bff}_0+\lrn{\bff}_1\lrn{\bfh}_0\Big).
\end{equation*}
As $\|\bfh\|_1=\theta<\delta/2\leq 1/4$,  we infer that
\begin{align*}
	\lrn{\widetilde \bff}_0	
	\leq & 4\Big(\lrn{\mathcal{N}}_0+\lrn{ \rR_N \bff}_0+\lrn{\bff}_1\lrn{\bfh}_0\Big).
\end{align*}
Here,  by  estimate \eqref{h_rnorm} and  Lemma \ref{Lem_trun} we readily get
\[ \lrn{ \rR_N \bff}_0\leq C_r \frac{\lrn{\bff}_{r}}{N^{r}},\qquad  \lrn{\bff}_1\lrn{\bfh}_0\leq C_{\varrho} N^{\varrho}\lrn{\bff}_1\,\| \bff,~\bfk\|_0,\]
for any $r\geq 0$. The term $\|\mathcal N\|_0$ can be estimated by \eqref{est_mathcalN}.  Thus, the desired estimate \eqref{tilbffnom0} follows immediately.

Now, we turn to investigate $\widetilde{\bfk}$. Observe that
 \begin{align*}
	\widetilde\bfk=H^{-1}\circ \KK\circ H-T_\alpha=(H^{-1}-\textup{id})\circ \KK\circ H+ \bfh+\bfk\circ H
\end{align*}
Then, using Proposition \ref{Apdix_pro1} we obtain a preliminary estimate for $\|\widetilde \bfk\|_0$ which will be useful below,
\begin{equation}\label{prior_bfk}
	\|\widetilde \bfk\|_0\leq \|H^{-1}-\textup{id}\|_0+ \|\bfh\|_0+\|\bfk\|_0\leq 2\|\bfh\|_0+\|\bfk\|_0.
\end{equation}

On the other hand, we deduce from the conjugacy  equation  $
H\circ \widetilde{\KK}=\KK\circ H
$ that
\begin{align*}
\begin{aligned}
\widetilde\bfk= &\bfk\circ H-\bfh\circ\widetilde\KK+\bfh\\
	=&\bfk-\Delta_\alpha \bfh+(\bfk\circ H-\bfk)-(\bfh\circ \widetilde\KK-\bfh\circ T_\alpha)\\
=& [\rS_N \bfk]+\rR_N \bfk+
	(\bfk\circ H-\bfk)-(\bfh\circ \widetilde\KK-\bfh\circ T_\alpha),
\end{aligned}	
\end{align*}
where for  the last line  we used the fact $\Delta_\alpha \bfh=\rS_N\bfk-[\rS_N\bfk]$. Then, for  $\widetilde\bfk=(\widetilde \bfk_1, \widetilde \bfk_2)$,
\begin{align}\label{witilk_11}
\widetilde \bfk_1=[\rS_N \bfk_1]+\widetilde\bfk_1',\quad\textup{with~} \widetilde\bfk_1'=\rR_N \bfk_1+
	(\bfk_1\circ H-\bfk_1)-(\bfh_1\circ \widetilde\KK-\bfh_1\circ T_\alpha),	
\end{align}
and
\begin{align}\label{witilk_22}
	 \widetilde \bfk_2=[\rS_N \bfk_2]+\rR_N \bfk_2+	(\bfk_2\circ H-\bfk_2)-(\bfh_2\circ \widetilde\KK-\bfh_2\circ T_\alpha).
\end{align}

For the term $\widetilde \bfk_2$, we apply  estimate   \eqref{prior_bfk} to obtain that
\begin{equation}\label{dsnbdons}
\begin{aligned}
	\lrn{\widetilde \bfk_2}_0\leq & \lrn{~[\rS_N \bfk_2]~}_0+\lrn{\rR_N \bfk_2}_0+\lrn{\bfk_2}_1\lrn{\bfh}_0+\lrn{\bfh_2}_1\lrn{\widetilde \bfk}_0\,\\
	= & \lrn{~[\bfk_2]-[\rR_N\bfk_2]~}_0+\lrn{\rR_N \bfk_2}_0+\lrn{\bfk_2}_1\lrn{\bfh}_0+\lrn{\bfh_2}_1 \,(2\|\bfh\|_0+\|\bfk\|_0)\\
	\leq & \lrn{~[\bfk_2]~}_0+2\lrn{\rR_N \bfk_2}_0+\lrn{\bfk_2}_1\lrn{\bfh}_0+\lrn{\bfh_2}_1 \,(2\|\bfh\|_0+\|\bfk\|_0)\\
	\leq & 2\|\bff, ~\bfk\|_1\|\bff, ~\bfk\|_0+2\lrn{\rR_N \bfk}_0+\lrn{\bfk}_1\lrn{\bfh}_0+2\lrn{\bfh}_1\lrn{\bfh}_0+\lrn{\bfh}_1\lrn{\bfk}_0,
\end{aligned}
\end{equation}
where  the last line used Lemma \ref{Lem_bracket} to estimate  $\|[\bfk_2]\|_0$.  
Now, applying \eqref{h_rnorm} to estimate $\|\bfh\|_1$ and $\|\bfh\|_0$ we can show that
\begin{equation*}
\lrn{\bfk}_1\lrn{\bfh}_0	\leq C_{\varrho}N^{\varrho}\lrn{\bfk}_1\| \bff,~\bfk\|_0, \quad \lrn{\bfh}_1\lrn{\bfh}_0\leq  C_{1,\varrho}N^{2\varrho}\lrn{\bff,~\bfk}_1\| \bff,~\bfk\|_0\,,
\end{equation*}
and
\[\lrn{\bfh}_1\lrn{\bfk}_0\leq C_{1,\varrho}N^{\varrho}\lrn{\bff,~\bfk}_1\| \bfk\|_0\,. \]
The term $\lrn{\rR_N \bfk}_0$ can be estimated using
Lemma \ref{Lem_trun}. Therefore,  \eqref{dsnbdons} reduces to
\begin{equation}\label{sglmlu}
	\lrn{\widetilde \bfk_2}_0\leq C_{r, \varrho}\left(
N^{2\varrho}\|\bff, ~\bfk\|_1\|\bff, ~\bfk\|_0+\frac{\|\bfk\|_{r}}{N^r}\right),\qquad\text{for any~} r\geq 0.
\end{equation}
This verifies the desired estimate \eqref{yokykyp}.

Using similar arguments, one can also show that
\begin{equation}\label{yokypopo}
	\lrn{\widetilde \bfk_1'}_0\leq C_{r, \varrho}\left(
N^{2\varrho}\|\bff, ~\bfk\|_1\|\bff, ~\bfk\|_0+\frac{\|\bfk\|_{r}}{N^r}\right),\qquad\text{for any~} r\geq 0.
\end{equation}
Thus, in order to complete  the  $C^0$ norm estimate of $\widetilde \bfk_1=[\rS_N \bfk_1]+\widetilde\bfk_1'$, it remains to control the average term $[\rS_N \bfk_1](y)$. In general, $[\rS_N\bfk_1]$  is only of  order one  without further information. This is the moment where we need the Lipschitz semi-conjugacy condition. 
Recall that $\KK$ is semi-conjugate to $R_\alpha$ via a Lipschitz semi-conjugacy $W:\TT\times \cD_\delta\to \TT$, which can be written as $W(x,y)=x+v(x,y)$ with $v\in\textup{Lip}(\TT\times \cD_\delta,\RR)$.
Define $\widetilde W(x,y):=W\circ H(x,y)$. It is Lipschitz continuous and 
\begin{equation}\label{new_W}
	\widetilde W(x,y)=x+\widetilde{v}(x,y),\quad\textup{with~} \widetilde{v}(x,y)=\bfh_1(x,y)+v\circ H(x,y).
\end{equation}
 Clearly, 
 $\widetilde \KK$ is  semi-conjugate to $R_\alpha$ via the semi-conjugacy $\widetilde W$, that is $\widetilde W\circ\widetilde \KK=R_\alpha\circ \widetilde W$ on $\TT\times\cD_{\widetilde\delta}$. 
 
 By \eqref{new_W}, the semi-conjugacy equation $\widetilde W\circ \widetilde \KK=R_\alpha\circ \widetilde W$ reduces to
\[
	x+\alpha+\widetilde \bfk_1+\widetilde v\circ \widetilde \KK=x+\widetilde v+\alpha,
\]
or equivalently, $[\rS_N \bfk_1]+\widetilde\bfk_1'+\widetilde v\circ \widetilde \KK-\widetilde v=0$.  It can  be rewritten as 
\begin{equation*}
\begin{split}
	[\rS_N \bfk_1](y)=-\widetilde\bfk_1'-\widetilde v (x+\alpha+[\rS_N \bfk_1]+\widetilde\bfk_1',y+\widetilde\bfk_2)
	+\widetilde v(&x+\alpha+[\rS_N \bfk_1],y)\\
	&-\widetilde v(x+\alpha+[\rS_N \bfk_1],y)+\widetilde v.
\end{split}	
\end{equation*}
Taking the average over $x\in\TT$  on both sides of the above identity, we get 
\begin{equation}\label{bra_SNk1}
\begin{split}
	[\rS_N \bfk_1](y)=
	-\int_{\TT}\widetilde\bfk_1'\,dx-\int_{\TT}\widetilde v (x+\alpha+[\rS_N \bfk_1]+\widetilde\bfk_1',y+\widetilde\bfk_2)-\widetilde v(x+\alpha+[\rS_N \bfk_1],y)\, dx
\end{split}	
\end{equation}
where we already used the fact that  for each fixed $y$,
\[\int_{\TT} \widetilde v(x+\alpha+[\rS_N \bfk_1](y),y)\, dx=\int_{\TT} \widetilde v(x,y)\, dx.\] 
Moreover,  $|\widetilde v(z)-\widetilde v(z')|$ $\leq$ $\mathfrak{\widetilde L}\cdot\textup{dist}(z,z')$ with some  Lipschitz bound $\mathfrak{\widetilde L}>1$ that satisfies 
\begin{equation}\label{ertm}
	\mathfrak{\widetilde L}\leq \|D\bfh_1\|_0+\mathfrak L\,(1+\|D\bfh\|_0)\leq \mathfrak L\,(1+2\|\bfh\|_1),
\end{equation}
as a consequence of \eqref{new_W} and
 $\mathfrak L>1$.  Then, we infer from  \eqref{bra_SNk1} and \eqref{ertm} that
 \[\lrn{~[\rS_N \bfk_1]~}_0\leq \lrn{\widetilde\bfk_1'}_0+\mathfrak{\widetilde{L}}  \cdot \lrn{\widetilde\bfk_1',~ \widetilde\bfk_2}_0\leq \lrn{\widetilde\bfk_1'}_0+2\mathfrak{L}  \cdot \lrn{\widetilde\bfk_1',~ \widetilde\bfk_2}_0,
	\]
where for the last inequality  we used  the fact $\|\bfh\|_1= \theta<\delta/2\leq 1/4.$
 This yields 
\begin{equation*}
\begin{aligned}
	\lrn{\widetilde \bfk_1}_0=\lrn{~[\rS_N \bfk_1]+\widetilde\bfk_1'}_0 
		\leq &\lrn{\widetilde\bfk_1'}_0+2\mathfrak{L}   \cdot \lrn{\widetilde\bfk_1',~ \widetilde\bfk_2}_0+\lrn{\widetilde\bfk_1'}_0\\
	\leq & \left(2+2\mathfrak{L} \right)  \cdot \lrn{\widetilde\bfk_1',~ \widetilde\bfk_2}_0\\
	\leq & 4 \mathfrak L \cdot \lrn{\widetilde\bfk_1',~ \widetilde\bfk_2}_0,
\end{aligned}
\end{equation*}
since $\mathfrak L>1$.  Thus, using  
 \eqref{sglmlu}--\eqref{yokypopo} the desired estimate \eqref{tildeknorm0} follows immediately. 
\end{proof}

\subsection{Proof of Proposition \ref{Pro_iterate}}
By what we have shown above, the desired $C^r$-estimate \eqref{hhhhnorm} of $\bfh$ follows from Lemma \ref{Lem_estim_h}.
The desired $C^0$-estimate \eqref{wtf0norm} of  $\lrn{\widetilde\bff, ~\widetilde\bfk}_0$ follows  from Lemma \ref{Lem_C0newerror}.
The estimate \eqref{newbound_L} for the Lipschitz bound $\widetilde{\mathfrak L}$ comes from \eqref{ertm}.

Thus, to complete the proof of Proposition \ref{Pro_iterate}, it remains to verify  estimate \eqref{wtf_rnorm} for  $\lrn{\widetilde\bff, ~\widetilde\bfk}_r$. More precisely, 
 $\widetilde\bff$  can be rewritten as
\begin{align*}
	\widetilde\bff=H^{-1}\circ \FF\circ H-U_0=&(H^{-1}-\textup{id})\circ \FF\circ H+\FF\circ H-U_0 \\
	=&(H^{-1}-\textup{id})\circ \FF\circ H+ U_0\circ H-U_0+\bff\circ H.
\end{align*}
Hence,
\begin{align}\label{FFminusU0}
\begin{aligned}
	\lrn{\widetilde\bff}_r\leq\|(H^{-1}-\textup{id})\circ \FF\circ H\|_r+ 2\|\bfh\|_r+\|\bff\circ H\|_r
\end{aligned}
\end{align}
According to Proposition \ref{Apdix_pro2}, for two smooth functions the $C^r$ norm of their composition can be controlled linearly if the $C^1$ norm of the two functions are bounded. 
We also point out that
\[(H^{-1}-\textup{id})\circ \FF\circ H(x+m,y)=(H^{-1}-\textup{id})\circ \FF\circ H(x,y),\]
\[\bff\circ H(x+m,y)=\bff\circ H(x,y).\]
for any $m\in\ZZ$, which, means that $(H^{-1}-\textup{id})\circ \FF\circ H$
 and $\bff\circ H$ are functions  on $\RR\times \cD_{\widetilde\delta}$ that are $\ZZ$-periodic in $x$. 
 
 Thus, to estimate $\|\widetilde\bff\|_r$
 it suffices to give the $C^r$ norm of  the right hand side  terms of \eqref{FFminusU0}
 on the  bounded domain $[0, 1]\times\cD_{\widetilde\delta}$. In fact, since  $\|\bfh\|_1$ and $\|\bff\|_1$ are bounded,  we infer from Proposition \ref{Apdix_pro2}  that
\begin{align*}
\begin{aligned}
	\lrn{(H^{-1}-\textup{id})\circ \FF\circ H}_r
	\leq & C_r\left(1+ \lrn{H^{-1}-\textup{id}}_r+\|\bff\|_r+\|\bfh\|_r\right),\\
	\lrn{\bff\circ H}_r\leq &C_r(1+\|\bff\|_r+\|\bfh\|_r).
\end{aligned}
\end{align*}
By Proposition \ref{Apdix_pro1},
\[\lrn{H^{-1}-\textup{id}}_r\leq C_r \|\bfh\|_r.\]
Together  with inequality  \eqref{h_rnorm}, we  finally get
\begin{align*}
	\lrn{\widetilde\bff}_r	
	\leq  C'_r\Big(1+\|\bfh\|_r+\|\bff\|_r\Big)
	\leq  C_{r,\varrho} \Big(1+N^{\varrho}\| \bff,~\bfk\|_{r}\Big)
\end{align*}
for every $r> 0$.
Next, we consider $\widetilde\bfk$. Observe that
\begin{align*}
	\widetilde\bfk=H^{-1}\circ \KK\circ H-T_\alpha=&(H^{-1}-\textup{id})\circ \KK\circ H+\KK\circ H-T_\alpha \\
	=&(H^{-1}-\textup{id})\circ \KK\circ H+ \bfh+\bfk\circ H
\end{align*}
Analogous to $\tilde{\bff}$,  one can show that 
\begin{align*}
	\lrn{\widetilde\bfk}_r \leq  C_{r,\varrho} \Big(1+N^{\varrho}\| \bff,~\bfk\|_{r}\Big)
\end{align*}
for every $r> 0$.	This verifies the desired estimate \eqref{wtf_rnorm}. 
Therefore, we finish the proof of Proposition \ref{Pro_iterate}.

\begin{Rem}[Lipschitz versus H\"older semi-conjugacy]\label{Rem_holdersemi}
We would like to say a little more on our Lipschitz semi-conjugacy condition,
which is only used to control the $C^0$-norm of the average term $[\rS_N\bfk_1]$. It seems possible to replace 
the Lipschitz semi-conjugacy condition by a H\"older one with a suitable H\"older exponent. More precisely, if one assumes that $\KK$ is semi-conjugate to $R_\alpha$ via a  $\beta$-H\"older semi-conjugacy, then by formula \eqref{bra_SNk1}  and the estimates  \eqref{sglmlu}--\eqref{yokypopo} we  would get
 \begin{align*}
 	\|[\rS_N \bfk_1]\|_0\ll &\|\widetilde\bfk_1'\|_0+ \|\widetilde\bfk_1',~ \widetilde\bfk_2\|_0^\beta \\
 	\ll & \left(
N^{2\varrho}\|\bff, ~\bfk\|_1\|\bff, ~\bfk\|_0+\frac{\|\bfk\|_{r}}{N^r}\right)+\left(
N^{2\varrho}\|\bff, ~\bfk\|_1\|\bff, ~\bfk\|_0+\frac{\|\bfk\|_{r}}{N^r}\right)^\beta
 \end{align*}
for every $r\geq 0$. Thus,  for the exponent $\beta$ greater than $\frac{1}{2}$ and close to $1$,   one may still obtain a higher-order estimate for $[\rS_N \bfk_1]$ by choosing suitably large $N$ at each KAM step. 

Anyway, our approach requires
the H\"older exponent $\beta$  to be close to $1$. It still does not give results for any  exponent $\beta\in (0,1] $, so we do not purse this direction in this paper.
\end{Rem}

\section{The KAM  iterative scheme}\label{section_Mproof}
In this section 
we  prove Theorem \ref{Thm_simplified} by using a KAM iterative scheme.   At each iteration step  we choose a  smoothing operator $\rS_{N_i}$ with  an appropriate $N_i>0$, and then apply Proposition \ref{Pro_iterate} to conjugate  the maps $\FF, \KK$  closer and closer to the linear maps $U_0, T_\alpha$.  The KAM technique ensures the rapid convergence of the iteration.

\begin{proof}[Proof of Theorem \ref{Thm_simplified}]	
Let $\delta\in (0, \frac{1}{2})$.
To begin the iterative process, we set up
\[\bff^{(0)}=\bff,\quad \bfk^{(0)}=\bfk;\]
\[\FF^{(0)}=U_0+\bff^{(0)},\quad \KK^{(0)}=T_\alpha+\bfk^{(0)}; \quad \bfh^{(0)}=0.\]
Here, the commuting maps 
\[ \FF^{(0)}, ~\KK^{(0)}:  \TT\times\cD_{\delta^{(0)}}\longrightarrow  \TT\times\RR\]
are diffeomorphisms from  $ \TT\times\cD_{\delta^{(0)}}$ onto their images, where $\delta^{(0)}=\delta$.  
By assumption,
$\KK^{(0)}$ is  semi-conjugate to $R_\alpha$ via a Lipschitz semi-conjugacy of the form $W^{(0)}(x,y)=x+v^{(0)}(x,y)$.  The function $v^{(0)}$ 
has a Lipschitz bound  $\mathfrak{L}^{(0)}=\mathfrak{L}>1$ on $\TT\times\cD_{\delta^{(0)}}$.

Then, at the $i$-th step ($i=1,2,\cdots$), with an appropriate large  $N_i>0$ we apply inductively Proposition \ref{Pro_iterate} to obtain $\bfh^{(i)}$, $\bff^{(i)}$, $\bfk^{(i)}$ such that
\[H^{(i)}=\textup{id}+\bfh^{(i)}\]
 \[\FF^{(i)}=\left(H^{(i)}\right)^{-1}\circ \FF^{(i-1)}\circ H^{(i)}=U_0+\bff^{(i)}\]
 \[\KK^{(i)}=\left(H^{(i)}\right)^{-1}\circ \KK^{(i-1)}\circ H^{(i)}=T_\alpha+\bfk^{(i)}\]    
where $\bfh^{(i)}\in C^\infty(\TT\times\cD_{\delta^{(i-1)}},\RR^{2})$, $\FF^{(i)}$ and $\KK^{(i)}$ are smooth diffeomorphisms from  $ \TT\times\cD_{\delta^{(i)}}$ onto their images, for some $\delta^{(i)}>0$, and $\bff^{(i)}, \bfk^{(i)}\in C^\infty(\TT\times\cD_{\delta^{(i)}},\RR^{2})$. 
In what follows, we introduce the notation
\begin{align*}
\cE_{i,r}\overset{\textup{def}}=\lrn{\bff^{(i)},~\bfk^{(i)}}_{C^r\left(\TT\times\cD_{\delta^{(i)}}\right)},\qquad
 \cU_{i,r}\overset{\textup{def}}=\lrn{\bfh^{(i)}}_{C^r\left(\TT\times\cD_{\delta^{(i-1)}}\right)}.  
\end{align*}
To ensure the convergence of the iteration process, at the $i$-th step ($i\geq 1$) we choose 
\begin{equation}\label{para_seq}
 N_{i}=\cE_{i-1,0}^{-\frac{1}{4(\varrho+1)}},
\end{equation}
Then, we infer from Proposition \ref{Pro_iterate} that for $i=1,2,\cdots,$
\begin{align}
\cU_{i,r}\leq & C_{r',r,\varrho}\, N_i^{r-r'+\varrho}\,\cE_{i-1,r'},\qquad \text{for~}r\geq r'\geq  0.\label{iterate_h_rnorm}\\
	\cE_{i,0} \leq & C_{r,\varrho} \cdot \mathfrak{L}^{(i-1)} \cdot\left(N_i^{2\varrho}\,\cE_{i-1,1}\cdot\cE_{i-1,0}+\frac{\cE^2_{i-1,\varrho+r+1}}{N_i^r}+\frac{\cE_{i-1,\varrho+r}}{N_i^r}\right),\qquad \text{for~} r\geq 0.\label{iterate_wtf0norm}\\
\cE_{i,r}\leq & C_{r,\varrho} \Big(1+N_i^{\varrho}\,\cE_{i-1,r}\Big),\qquad \text{for~}r> 0.\label{iterate_wtf_rnorm}
\end{align}
and
\begin{align}\label{form_delta_i}
	\delta^{(i)}=\delta^{(i-1)}-2 \cU_{i,1}-\cE_{i-1,0}
\end{align}
 Moreover,  by  \eqref{newbound_L}, $\KK^{(i)}$ is  semi-conjugate to $R_\alpha$ via 
a Lipschitz semi-conjugacy $W^{(i)}(x,y)=x+v^{(i)}(x,y)$, where 
 $v^{(i)}$ has a Lipschitz bound  $\mathfrak{L}^{(i)}$ satisfying
\begin{equation}\label{mathfrakL_i}
	\mathfrak{L}^{(i)}\leq \mathfrak{L}^{(i-1)}\,(1+2\cU_{i,1}).
\end{equation}

Set $\mu=15(\varrho+1)$. The following result holds.
\begin{Lem}\label{Lem_induc_ineq}
Assume that $\cE_{0,\mu}=\|\bff^{(0)},~\bfk^{(0)}\|_\mu$ is sufficiently small, then for all $i\geq 1$,
	\begin{align}\label{3fchi}
\cE_{i,0}\leq \cE^{\frac{5}{4}}_{i-1,0}~,\quad
 \cE_{i,\mu}\leq \cE_{i,0}^{-1}~,\quad \cU_{i,1}\leq\cE^{\frac{1}{2}}_{i-1,0},\qquad \delta^{(i)}\geq \frac{\delta}{2}+\frac{\delta}{2^{i+1}}.
\end{align}
\end{Lem}
\begin{proof}
Note that by the interpolation inequalities  (see Lemma \ref{cor_intpest}),  we get
\begin{equation}\label{INTP_inequ}
	\cE_{i,1}\leq C_\mu\,\cE_{i,0}^{1-\frac{1}{\mu}} \,\cE_{i,\mu}^{\frac{1}{\mu}},\qquad\textup{for all}~ i.
\end{equation}
According to \eqref{para_seq}--\eqref{mathfrakL_i}, it is easy to find that the inequalities in \eqref{3fchi} are true for the first step $i=1$, provided that $\cE_{0,\mu}$ is suitably small.

Suppose inductively that all inequalities in \eqref{3fchi} hold for $1 ,\cdots,i$.
Then, we will check  these estimates for the $(i+1)$-th step.

Since $N_{i+1}=\cE_{i,0}^{-\frac{1}{4(\varrho+1)}}$ and \eqref{INTP_inequ} holds,   
using inequality \eqref{iterate_wtf0norm} with $r=\mu-(\varrho+1)$ we  obtain 
\begin{align}
	\cE_{i+1,0}\leq & C_{\mu,\varrho}\cdot\mathfrak{L}^{(i)} \cdot \left(N_{i+1}^{2\varrho}\cE_{i,1}\cdot\cE_{i,0}+\frac{\cE^2_{i,\mu}}{N_{i+1}^{\mu-\varrho-1}}+\frac{\cE_{i,\mu-1}}{N_{i+1}^{\mu-\varrho-1}}    \right)\nonumber\\
	 \leq& C'_{\mu,\varrho} \cdot\mathfrak{L}^{(i)}\cdot\left(N_{i+1}^{2\varrho}\cE_{i,0}^{2-\frac{2}{\mu}} +\frac{\cE_{i,0}^{-2}}{N_{i+1}^{\mu-\varrho-1}}+\frac{\cE_{i,0}^{-1}}{N_{i+1}^{\mu-\varrho-1} }\right)\nonumber\\
	 =&C'_{\mu,\varrho} \cdot\mathfrak{L}^{(i)} \cdot\left(\cE_{i,0}^{2-\frac{2}{\mu}-\frac{\varrho}{2(\varrho+1)}} +\cE_{i,0}^{\frac{\mu-\varrho-1}{4(\varrho+1)}-2}+\cE_{i,0}^{\frac{\mu-\varrho-1}{4(\varrho+1)}-1} \right).\label{dvfyvd}
\end{align}
By \eqref{mathfrakL_i}, we derive inductively that
\begin{align*}
	\mathfrak{L}^{(i)}\leq\mathfrak{L}\,\prod_{t=1}^{i}(1+2\cU_{t,1})\leq  \mathfrak{L}\,\prod_{t=1}^{i}(1+2\cE_{t-1,0}^{\frac{1}{2}})
	\leq \mathfrak{L}\,\prod_{t=1}^{i}\left(1+2\cE_{0,0}^{\frac{1}{2}\left(\frac{5}{4}\right)^{t-1}}\right)
	\leq  C\,\mathfrak{L},
\end{align*}
where $C>1$ is a constant independent of $i$ provided that $\cE_{0,0}<1/2$. Observe  that $\mu=15(\varrho+1)$, then
\[2-\frac{2}{\mu}-\frac{\varrho}{2(\varrho+1)}>\frac{3}{2},\qquad \frac{\mu-\varrho-1}{4(\varrho+1)}-2=\frac{14(\varrho+1)}{4(\varrho+1)}-2= \frac{3}{2}.\]
Substituted into \eqref{dvfyvd}, we obtain 
\begin{equation}\label{cE_i1}
\cE_{i+1,0}\leq C'_{\mu,\varrho} \, C\,\mathfrak{L}\, \left(\cE_{i,0}^{\frac{3}{2}} +\cE_{i,0}^{\frac{3}{2}} +\cE_{i,0}^{\frac{5}{2}} \right)\leq\cE_{i,0}^{\frac{5}{4}}.	\end{equation}

Applying \eqref{iterate_wtf_rnorm} with  $r=\mu$, it follows that
\begin{align}\label{rho}
	\cE_{i+1,\mu}\leq C_{\mu, \varrho} \left(1+N_{i+1}^{\varrho}\cE_{i,\mu} \right)\leq 2C_{\mu,\varrho}\cE_{i,0}^{-1-\frac{\varrho}{4(\varrho+1)}}\leq\cE_{i,0}^{-\frac{5}{4}}\leq \cE^{-1}_{i+1,0}.
\end{align}
Here, for the last inequality we used \eqref{cE_i1}.

Next,   applying inequality  \eqref{iterate_h_rnorm}  with $r=1$ and $r'=0$,  we have
\begin{align}
	\cU_{i+1,1}	&\leq C_{0,1,\varrho}\, N_{i+1}^{\varrho+1}\cE_{i,0}  \leq C_{0,1,\varrho}\, \cE^{1-\frac{\varrho+1}{4(\varrho+1)}}_{i,0}\leq \cE_{i,0}^{\frac{1}{2}}.
\end{align}

Finally, by \eqref{form_delta_i} it follows that
\begin{align}
	\delta^{(i+1)}=\delta^{(i)}-2 \cU_{i+1,1}-\cE_{i,0}&\geq \frac{\delta}{2}+\frac{\delta}{2^{i+1}}-2\cE_{i,0}^{\frac{1}{2}}-\cE_{i,0}\nonumber\\
	&\geq\frac{\delta}{2}+\frac{\delta}{2^{i+1}}-3\cE_{0,0}^{\frac{1}{2}\left(\frac{5}{4}\right)^{i}}\geq \frac{\delta}{2}+\frac{\delta}{2^{i+2}}	\label{delta_i1}
\end{align}
as long as $\cE_{0,0}<c\cdot\delta$ for some small constant $c>0$.

Combining \eqref{cE_i1}--\eqref{delta_i1},
we thus verify \eqref{3fchi} for $(i+1)$ in place of $i$.
This proves Lemma \ref{Lem_induc_ineq}.
\end{proof}

Now, let us proceed with the proof of Theorem \ref{Thm_simplified}.
By Lemma \ref{Lem_induc_ineq}, as long as $\cE_{0,\mu}$ is sufficiently small,  the following sequences 
\[\|\bff^{(i)},~\bfk^{(i)}\|_0\leq \cE_{0,0}^{\left(\frac{5}{4}\right)^i},\quad \|\bfh^{(i)}\|_1\leq \cE_{0,0}^{\frac{1}{2}\left(\frac{5}{4}\right)^{i-1}} \]
converge rapidly to zero. Also,  $\delta^{(i)}\to\frac{\delta}{2}$. Thus,
 this rapid convergence ensures  that 
as $l\to\infty$, the composition 
\[\mathcal{H}_l =H^{(1)}\circ\cdots\circ H^{(l)}\] converges in the $C^1$ topology to  some $\mathcal{H}_\infty$ which is a $C^1$ diffeomorphism from $\TT\times\cD_{\frac{\delta}{2}}$ onto its image, for which the 
 following conjugacy equations hold 
\[ \FF\circ \mathcal{H}_\infty=\mathcal{H}_\infty \circ U_0,\qquad \KK\circ \mathcal{H}_\infty=\mathcal{H}_\infty\circ T_\alpha.\]

Now, it remains to show that the $C^1$ limit solution $\mathcal{H}_\infty$ is also of class $C^s$ for every $s>1$. In fact,  just as shown in \cite{Zeh_generalized1}, 
this can be achieved by making full use of the
 interpolation inequalities.  
 
More precisely, we first observe that for  any $t>0$, applying \eqref{iterate_wtf_rnorm} with $r=t$ we get
\[
	\cE_{i,t}\leq  C_{t,\varrho} \Big(1+N_i^{\varrho}\cE_{i-1,t}\Big),
\]
for some constant $C_{t,\varrho}> 1$.
In light of the choice of $N_i$ (see \eqref{para_seq}), it follows  that
\[1+\cE_{i,t}\leq C_{t,\varrho} \,\cE_{i-1,0}^{-\frac{1}{4}}\Big(1+\cE_{i-1,t}\Big),\]
from which we derive inductively that
\begin{align*}
	\cE_{i,t}\leq \left(1+\cE_{0,t} \right)\prod_{j=0}^{i-1}\left(C_{t,\varrho} \,\cE_{j,0}^{-\frac{1}{4}}\right)&\leq  \left(1+\cE_{0,t} \right) C^i_{t,\varrho}\prod_{j=0}^{i-1} \cE_{0,0}^{-\frac{1}{4}\left(\frac{5}{4}\right)^j}\\
	&\leq  M_t\cdot C^i_{t,\varrho}\cdot \cE^{-\left(\frac{5}{4}\right)^i}_{0,0},
\end{align*}
with $M_t= (1+\cE_{0,t})$.

Now, for any fixed $s>1$, we choose $t=4s$. The interpolation  inequalities  (Lemma \ref{cor_intpest}) imply
\begin{align*}
	\cE_{i,s}\leq  C_{s}\,  \cE_{i,t}^{\frac{1}{4}}\cdot \cE_{i,0}^{\frac{3}{4}}\leq
	C_s\left(M_t\, C^i_{t,\varrho}\right)^{\frac{1}{4}}\,   \cE_{0,0}^{-\frac{1}{4}\left(\frac{5}{4}\right)^i}\cdot \cE_{0,0}^{\frac{3}{4}\left(\frac{5}{4}\right)^i}\leq  C_s\, M_t\, C_{t,\varrho}^{i}\cdot \cE^{\frac{1}{2}\left(\frac{5}{4}\right)^i}_{0,0}\, .
\end{align*}
Then, applying  \eqref{iterate_h_rnorm} with $r=r'=s$ yields 
\begin{align*}
\cU_{i+1,s}\leq  C_{s,\varrho}\, N_{i+1}^{\varrho}\cE_{i,s}= C_{s,\varrho} \cE_{i,0}^{-\frac{\varrho}{4(\varrho+1)}}\cE_{i,s}
\leq &  C_{s,\varrho} C_s\,M_t\, C^{i}_{t,\varrho}\cdot \cE_{0,0}^{-\frac{\varrho}{4(\varrho+1)}\left(\frac{5}{4}\right)^i}\cE^{\frac{1}{2}\left(\frac{5}{4}\right)^i}_{0,0}\nonumber\\
\leq &L \cdot b^i\cdot\cE^{\frac{1}{4}\left(\frac{5}{4}\right)^i}_{0,0}
\end{align*}
where the constants $L=C_{s,\varrho} C_s\,M_t$  and $b=C_{t,\varrho}>1$, with $t=4s$. 
Observe that although  $ b^i $ grows exponentially, the term  $\cE^{\frac{1}{4}\left(\frac{5}{4}\right)^i}_{0,0}$ decays super-exponentially as $i\to\infty$. Hence,  
\[\cU_{i+1,s}=\|\bfh^{(i+1)}\|_{s}\]
still converges rapidly to zero as $i\to\infty$.
This implies the convergence of
the sequence $\mathcal{H}_l$ in the $C^s$ topology and the limit is exactly $\mathcal{H}_\infty$. Therefore,
the limit  $\mathcal{H}_\infty$ is a $C^\infty$ diffeomorphism of $\TT\times\cD_{\frac{\delta}{2}}$ onto its image. 
This finishes the proof.
\end{proof}

Now that Theorem \ref{Thm_simplified} has been proved, by what we have shown in Section \ref{section_renorm} it also implies Theorem \ref{MainR1}. 

\noindent{\bf Acknowledgments.}
 We sincerely thank the anonymous referees for their comments and valuable suggestions on improving our results. Our work was supported by 
Swedish Research Council VR grant 2015-04644, 
 VR grant 2019-04641, and the Wallenberg Foundation grant for international postdocs 2020. 

\appendix

\section{}

For an open set $D\subset\RR$ and $\delta>0$, we denote by 
 $D_\delta=\{y\in\RR:~\textup{dist}(y, D)<\delta\}.$
 Obviously, $D_\delta$ is convex if $D$ is convex.
We have the following elementary fact on the inverse function. 
 \begin{Pro}\label{Apdix_pro1}
 Let $D\subset \RR$ be a  bounded open interval, and
  \[\Phi=\textup{id}+\phi=(x+\phi_1(x,y),y+\phi_2(x,y))\] be a smooth map defined on $\TT\times D_\delta$. 
Denote $\theta_1=\|\phi\|_{1}$. 
Suppose that
 \[\theta_1<\delta\leq \frac{1}{2}.\]
 Then,  $\Phi$ has a smooth inverse map $\Phi^{-1}$ defined on $\TT\times D_{\delta-\theta_1}$, which satisfies
 \[\lrn{\Phi^{-1}-\textup{id}}_{0}\leq  \lrn{\phi}_{0},
  \qquad \lrn{\Phi^{-1}-\textup{id}}_{r}\leq d_r\cdot \lrn{\phi}_{r}\]
 for $r> 0$, where the constant $d_r$  depends on $r$.
 \end{Pro}
 \begin{proof}
 We first claim that $\Phi$ is injective on $\TT\times D_{\delta}$. 
 Denote $z=(x,y)$. Consider two points $z$ and $z'$ in $\TT\times D_{\delta}$ with $\Phi(z)=\Phi(z')$. Then 
 \[z-z'=\phi(z')-\phi(z).\] Since $D_\delta$ is convex,  the segment $(1-t)z+tz'$, $t\in[0,1]$ is strictly contained in $\TT\times D_{\delta}$. So, using the mean value theorem,
 \[   \|z-z'\|\leq \|D\phi\|\,\|z-z'\|\leq\theta_1 \|z-z'\|< \frac{1}{2}\|z-z'\|.\]
This implies  $z=z'$. Therefore, $\Phi$ is injective on $\TT\times D_{\delta}$.

Since $\|\Phi-\textup{id}\|_0\leq \theta_1$,
by elementary arguments from degree theory the image of $\TT\times D_{\delta}$ under $\Phi$ covers $\TT\times D_{\delta-\theta_1}$. Consequently, $\Phi$ has a smooth inverse on   $\TT\times D_{\delta-\theta_1}$.

 The $C^r$ norm estimate can be achieved by using interpolation estimates (cf. \cite[Lemma 2.3.6]{Hamil_1982}).
 \end{proof}
 
For  two smooth functions, the $C^r$ norm of their composition can be controlled linearly provided that the $C^1$ norm of these two functions are bounded.
  \begin{Pro}\label{Apdix_pro2}
 Let  $\Phi_1 : B^m\to  B^n $  and  $\Phi_2: B^l\to B^m$ be $C^\infty$ functions where $B^\iota\subset \RR^\iota$, $\iota=m, n,l$ are bounded  domains. Assume that the $C^1$ norms $\|\Phi_1\|_1\leq M$ and $\|\Phi_2\|_1\leq M$, then the composition $\Phi_1\circ \Phi_2$ satisfies: for all $r\geq 0$,
 \begin{align*}
 	\|\Phi_1\circ \Phi_2\|_r \leq C_{M,r}\left(1+\|\Phi_1\|_r+\|\Phi_2\|_r\right),
 \end{align*}
where  the constant $C_{M,r}$ depends  on $M$ and $r$.
 \end{Pro}

It is proved in \cite[Lemma 2.3.4]{Hamil_1982}. See also \cite{dlL_Obaya_1999} for general domains in Banach spaces.

\section{Higher-dimensional maps}
We remark here that the whole proof of Theorem \ref{Thm_simplified} would go through in higher dimensions $\TT^d\times\RR^d$, $d\geq 2$. However, the intersection property in higher dimensions is not satisfied even by the unperturbed maps (we explain this below).  It is possible that with a property weaker than the intersection property the theorem can be extended to higher dimensions. 

Recall that any exact symplectic map of $\TT\times\RR$ satisfies the intersection property.  
In this section, we  show that in the case of higher-dimensional maps of $\TT^d\times\RR^d$ ($d\geq 2$), there are even exact symplectic maps that do not satisfy the intersection property.

For simplicity, here we only consider the maps of $\TT^2\times\RR^2$.
Recall that a map $F(x,y): \TT^2\times\RR^2\to \TT^2\times\RR^2$ is said to satisfy the \emph{intersection property} if each $2$-dimensional torus  close to the ``horizontal'' torus $\{y=const\}$ intersects its image under $F$.

\begin{ex}
	Let $F_0(x,y)=(x+y, y)$, where $x=(x_1,x_2)\in\TT^2$ and $y=(y_1, y_2)\in\RR^2$.
	Obviously, $F_0$ is exact symplectic. However, we claim that $F_0$ does not satisfy the intersection property
\end{ex}

Assume by contradiction that $F_0$ satisfies the intersection property, then  for each 2-dimensional torus of the form $y=\psi(x)$ where the function
$\psi:\TT^2\to\RR^2$ is close to a constant vector $y_0\in\RR^2$, it satisfies 
\begin{equation*}
	\Big\{(x, \psi(x) ): x\in\TT^2\Big\}\bigcap \Big\{(x+\psi(x), \psi(x) ): x\in\TT^2\Big\}\neq\emptyset.
\end{equation*}
as a result of $F_0(x,y)=(x+y,y)$. 
 This is equivalent to saying the following equation
 \begin{equation}\label{dqqeeqe}
 	 \psi(x+\psi(x))=\psi(x) \quad \textup{has at least one solution~}x\in\TT^2.
 \end{equation}

In particular, we consider  the 2-dimensional torus $\cG=\{(x,\psi(x)) : x\in \TT^2\}$ where $\psi=(\psi_1,\psi_2):\TT^2=\RR^2/\ZZ^2\to\RR^2$ is given by \[\psi_1(x_1,x_2)=\frac{1}{2}+\delta\sin 2\pi x_1,\quad\psi_2(x_1,x_2)=\delta \cos 2\pi x_1,\]
 and $\delta\in(0,\frac{1}{2\pi})$ is sufficiently small. Note that  $\cG$ is close to the 
 ``horizontal''    torus $\TT^2\times\{y=(\frac{1}{2}, 0)\}$.
Then,   \eqref{dqqeeqe} implies that  the following system of equations  admits solutions,
\begin{align*}
\left\{
\begin{array}{rrr}
	\frac{1}{2}+\delta\sin2\pi(x_1+\frac{1}{2}+\delta\sin 2\pi x_1)&=&\frac{1}{2}+\delta\sin 2\pi x_1\\
	 	\delta\cos2\pi(x_1+\frac{1}{2}+\delta\sin2\pi x_1)&=&\delta\cos 2\pi x_1
\end{array}
\right.
\end{align*}
As $\delta>0$, this implies that the following two functions $g(x_1)$ and $h(x_1)$  have common zeros,
\begin{align*}
	g(x_1):=&\sin2\pi(x_1+\frac{1}{2}+\delta\sin 2\pi x_1)-\sin 2\pi x_1\\ h(x_1):=&\cos2\pi(x_1+\frac{1}{2}+\delta\sin2\pi x_1)-\cos2\pi x_1
\end{align*}
Observe that  $g(x_1)$  has only two zeros (mod $1$).
Indeed, as $\delta>0$ is sufficiently  small, it is easy to check that for all $x_1\in(0,\frac{1}{2})$,  we get $x_1+\frac{1}{2}+\delta\sin 2\pi x_1\in (\frac{1}{2},1)$, which implies 
\[\sin2\pi(x_1+\frac{1}{2}+\delta\sin 2\pi x_1)<0,\qquad \sin 2\pi x_1>0\qquad \text{for all~} x_1\in(0,\frac{1}{2}),\]
and hence $g(x_1)<0$ for $x_1\in(0,\frac{1}{2})$. Moreover, since $g(x_1)=-g(1-x_1)$, we find that $g(x_1)>0$ for all $x_1\in(\frac{1}{2},1)$. Thus, $g(x_1)$ has only two zeros $x_1=0$ (mod 1) and $x_1=\frac{1}{2}$ (mod 1).
But, $h(0)=-2$ and $h(\frac{1}{2})=2$. This is a contradiction.

In conclusion,   $F_0(\cG)\cap \cG=\emptyset$, and thus $F_0$ has no intersection property. As a consequence, we have the following result.
\begin{Cor}
	For every (exact symplectic) map $F:\TT^2\times\RR^2\to \TT^2\times\RR^2$ that is sufficiently close to $F_0(x,y)=(x+y,y)$ in the $C^1$ topology, $F$   does not satisfy the intersection property.
\end{Cor}
\begin{proof}
	By what we have shown above, there exists a 2-dimensional torus $\cG$ such that $F_0(\cG)\cap \cG=\emptyset$. Then, for every $F$ sufficiently close to $F_0$, and every torus $\cG'$ sufficiently close to $\cG$ in the $C^1$ topology, we have $F(\cG')\cap \cG'=\emptyset$.	 
\end{proof}

\newcommand{\etalchar}[1]{$^{#1}$}

\end{document}